%% file: paper.tex
\pdfoutput=1
\documentclass[preprint]{elsarticle}
\usepackage{amsfonts}
\usepackage{listings}
\usepackage{color}
\usepackage{courier}
\usepackage{pdfpages}
\usepackage[T1]{fontenc}
\usepackage[latin9]{inputenc}
\usepackage{geometry}
\usepackage{amsmath}
\usepackage{amssymb}
\usepackage{amsthm}
\usepackage{soul}
\usepackage{url}
\usepackage{float}
\usepackage{subcaption}
\usepackage{multirow}
\usepackage{graphicx}
\usepackage{caption}
\usepackage{enumitem}
\setlist[enumerate]{listparindent=0.5in}
\usepackage{nomencl}
\makenomenclature
\usepackage{titlesec}
\usepackage{algpseudocode}
\usepackage{algorithm}

\usepackage[algo2e, linesnumbered]{algorithm2e}
\SetKwInOut{Input}{Input}
\SetKwInOut{Output}{Output}
\usepackage{booktabs}
\usepackage{hyperref}
\usepackage{cleveref}
\crefname{figure}{Figure}{Figures}

\include{pauls-macros}

\date{\today}
\begin{document}
\title{\textbf{Embedded Ridge Approximations\footnote{\textcopyright\; 2020. This manuscript version is made available under the CC-BY-NC-ND 4.0 license \url{http://creativecommons.org/licenses/by-nc-nd/4.0/}} } }
\date{}

\newcommand{\norm}[2]{\ensuremath{\left\lVert#1\right\rVert_{#2}}}
\newcommand{\mb}[1]{\mathbf{#1}}
\newcommand{\myRe}{\mathbb{R}}
\newcommand{\myEx}{\mathbb{E}}
\newcommand{\V}{\text{Var}}
\linespread{1.0}
\let\oldnl\nl
\newcommand{\nonl}{\renewcommand{\nl}{\let\nl\oldnl}}
\newtheorem{theorem}{Theorem}[section]
\newtheorem{lemma}{Lemma}[section]
\newtheorem{definition}{Definition}[section]
\thispagestyle{plain}

\author{Chun Yui Wong\corref{cor1}%
\fnref{fn1}}
\ead{cyw28@cam.ac.uk}
\author{Pranay Seshadri\fnref{fn2}}

\author{Geoffrey T. Parks\fnref{fn3}}
\author{Mark Girolami \fnref{fn4}}

\cortext[cor1]{Corresponding author}
\fntext[fn1]{PhD student, Department of Engineering, University of Cambridge.}
\fntext[fn2]{Research Fellow, Department of Mathematics (Statistics Section), Imperial College London.}
\fntext[fn3]{Reader, Department of Engineering, University of Cambridge.}
\fntext[fn4]{Sir Kirby Laing Chair of Civil Engineering, Department of Engineering, University of Cambridge, and Strategic Director, The Alan Turing Institute.}

\begin{abstract}
Many quantities of interest (qois) arising from differential-equation-centric models can be resolved into functions of scalar fields. Examples of such qois include the lift over an airfoil or the displacement of a loaded structure; examples of corresponding fields are the static pressure field in a computational fluid dynamics solution, and the strain field in the finite element elasticity analysis. These scalar fields are evaluated at each node within a discretised computational domain. In certain scenarios, the field at a certain node is only weakly influenced by far-field perturbations; it is likely to be strongly governed by local perturbations, which in turn can be caused by uncertainties in the geometry. One can interpret this as a strong anisotropy of the field with respect to uncertainties in prescribed inputs. We exploit this notion of \emph{localised scalar-field influence} for approximating global qois, which often are integrals of certain field quantities. We formalise our ideas by assigning ridge approximations for the field at select nodes. This \emph{embedded ridge approximation} has favorable theoretical properties for approximating a global qoi in terms of the reduced number of computational evaluations required. Parallels are drawn between our proposed approach, active subspaces and vector-valued dimension reduction. Additionally, we study the \emph{ridge directions} of adjacent nodes and devise algorithms that can recover field quantities at selected nodes, when storing the ridge profiles at a subset of nodes---paving the way for novel reduced order modeling strategies. Our paper offers analytical and simulation-based examples that expose different facets of embedded ridge approximations.
\end{abstract}

\maketitle

\input{sec-1}
\input{sec-2}
\input{sec-3}
\input{sec-4}

\subsection*{Acknowledgements}
CYW acknowledges financial support from the Cambridge Trust, Jesus College, Cambridge, and the Alan Turing Institute. PS was funded through a Rolls-Royce postdoctoral fellowship. The authors would like to thank the anonymous reviewer for useful feedback, which helped improve the manuscript.

\appendix
\renewcommand{\thesection}{\Alph{section}}

\section{Proof of \Cref{thm:sample_bound}} \label{sec:sample_bound_proof}
By the matrix Bernstein inequality, it can be shown that \eqref{equ:M_cond} implies that \cite[Remark 6.5]{tropp2012user-friendly}
\begin{equation}
\myEx\norm{\frac{1}{M} \sum_{m=1}^M \nabla g_{i}^{(m)} \nabla g_{j}^{(m)T} - \myEx[\nabla g_i \nabla g_j^T]}{2} \leq (\epsilon + \epsilon^2) \norm{\myEx[\nabla g_i \nabla g_j^T]}{2},
\end{equation}
We can then write
\begin{align*}
\myEx\norm{\widehat{\mC}(h) - \mC(h)}{2} &= \myEx\norm{\frac{1}{M} \sum_{m=1}^M \sum_{ij} \omega_i \omega_j \widehat{\mW}_{i} \nabla g_{i}^{(m)} \nabla g_{j}^{(m)T} \widehat{\mW}_{j}^T - \sum_{ij} \omega_i \omega_j \mW_i \myEx[\nabla g_i \nabla g_j^T] \mW_j^T}{2}\\
&= \myEx\norm{\sum_{ij} \omega_i \omega_j \mE_{ij}}{2}\\
&\leq \sum_{ij} |\omega_i \omega_j| \myEx\norm{\mE_{ij}}{2},
\end{align*}
where 
\begin{align*}
\mE_{ij} &= \frac{1}{M} \sum_{m=1}^M \widehat{\mW}_{i} \nabla g_{i}^{(m)} \nabla g_{j}^{(m)T} \widehat{\mW}_{j}^T - \mW_i \myEx[\nabla g_i \nabla g_j^T] \mW_j^T\\
&= \frac{1}{M} \sum_{m=1}^M \left( \widehat{\mW}_{i} \nabla g_{i}^{(m)} \nabla g_{j}^{(m)T} \widehat{\mW}_{j}^T - \mW_{i} \nabla g_{i}^{(m)} \nabla g_{j}^{(m)T} \widehat{\mW}_{j}^T  \right)\\
&\quad+ \frac{1}{M} \sum_{m=1}^M \left( \mW_{i} \nabla g_{i}^{(m)} \nabla g_{j}^{(m)T} \widehat{\mW}_{j}^T - \mW_{i} \nabla g_{i}^{(m)} \nabla g_{j}^{(m)T} \mW_{j}^T \right) \\
&\quad+ \frac{1}{M} \sum_{m=1}^M \left( \mW_{i} \nabla g_{i}^{(m)} \nabla g_{j}^{(m)T} \mW_{j}^T \right) - \mW_i \myEx[\nabla g_i \nabla g_j^T] \mW_j^T\\
&=  \left(\widehat{\mW}_{i} - \mW_i\right)\left( \frac{1}{M} \sum_{m=1}^M \nabla g_{i}^{(m)} \nabla g_{j}^{(m)T}\right) \widehat{\mW}_{j}^T\\
&\quad+ \mW_i \left( \frac{1}{M} \sum_{m=1}^M \nabla g_{i}^{(m)} \nabla g_{j}^{(m)T}\right) \left(\widehat{\mW}_{j}^T - \mW_j^T\right)\\
&\quad+ \mW_i \left(\frac{1}{M} \sum_{m=1}^M \nabla g_{i}^{(m)} \nabla g_{j}^{(m)T} - \myEx[\nabla g_i \nabla g_j^T] \right)\mW_j^T.
\end{align*}
Note that
\begin{equation}
\myEx\norm{\frac{1}{M} \sum_{m=1}^M \nabla g_{i}^{(m)} \nabla g_{j}^{(m)T}}{2} \leq \frac{1}{M} \sum_{m=1}^M \myEx\norm{\nabla g_{i}^{(m)} \nabla g_{j}^{(m)T}}{2} \leq L^2
\end{equation}
because $\nabla g_i^{(m)}$ are identically distributed copies of $\nabla g_i$ and whose norms are upper bounded by $L$. Using the triangle inequality and sub-multiplicativity of the norm, we get
\begin{equation}
\sum_{ij} |\omega_i \omega_j| \myEx\norm{\mE_{ij}}{2} \leq C\left(2\eta_M L^2 + \left(\epsilon + \epsilon^2\right)L^2\right)
\end{equation}
from which the theorem follows.

\section{Proof of \Cref{thm:subspace_perturb}}
\label{sec:subspace_proof}
First, we prove a lemma establishing a connection between perturbation of subspaces and a perturbation in the associated basis matrices.

\begin{lemma} \label{lem:inner_prod}
Let $\mathcal{S}$ be an $r$-dimensional subspace of $\myRe^d$, and $\widetilde{\mathcal{S}}$ be an equidimensional perturbation of $\mathcal{S}$. Then, $dist(\mathcal{S}, \widetilde{\mathcal{S}})$ is bounded above if and only if there exists $\mW, \widetilde{\mW} \in \myRe^{d\times r}$ with orthonormal columns such that $\vw_i^T \widetilde{\vw}_i$ is bounded below for all $i=1,...,r$, where $\vw_i$ is the $i$-th column of $\mW_i$ (and similarly for perturbed quantities).
\end{lemma}

\begin{proof}
($\Rightarrow$) Choose $\{\vw_1,...,\vw_r\}$ and $\{\widetilde{\vw}_1,...,\widetilde{\vw}_r\}$ to be a set of principal vectors of $\mathcal{S}$ and $\widetilde{\mathcal{S}}$ respectively. Then, by construction, we have 
\begin{equation}
\vw_i^T \widetilde{\vw}_i = \cos(\theta_i),
\end{equation}
where $\theta_i$ is the $i$-th principal angle, with $0\leq \theta_1 \leq ... \leq \theta_r \leq \pi/2$. Thus, $\vw_i^T \widetilde{\vw}_i$ is bounded below by $\cos(\theta_r)$, the cosine of the largest principal angle. The result follows from the fact that $\cos(\theta_r) = \sqrt{1-dist(\mathcal{S}, \widetilde{\mathcal{S}})^2}$ \cite[Section 6.4.3]{golub2013matrix}.

($\Leftarrow$) Note that 
\begin{equation}
dist(\mathcal{S}, \widetilde{\mathcal{S}}) = \norm{\mW\mW^T - \widetilde{\mW}\widetilde{\mW}^T}{2},
\end{equation}
for any $\mW, \widetilde{\mW} \in \myRe^{d\times r}$ whose columns define orthonormal bases for $\mathcal{S}$ and $\widetilde{\mathcal{S}}$ respectively. That is, the distance does not depend on the basis chosen for the orthogonal projector. Then, we can write
\begin{align}
\begin{split}
\norm{\mW\mW^T - \widetilde{\mW}\widetilde{\mW}^T}{2} &= \norm{\sum_{i=1}^r \vw_i\vw_i^T - \widetilde{\vw}_i\widetilde{\vw}_i^T}{2}\\
&\leq \sum_{i=1}^r \norm{\vw_i\vw_i^T - \widetilde{\vw}_i\widetilde{\vw}_i^T}{2}\\
&= \sum_{i=1}^r \sqrt{1-\left(\vw_i^T \widetilde{\vw_i}\right)^2}.\\
\end{split}
\end{align}
Hence, if $\vw_i^T \widetilde{\vw}_i \geq \cos(\theta_r)$ then
\begin{equation}
dist(\mathcal{S}, \widetilde{\mathcal{S}}) \leq \sum_{i=1}^r \sqrt{1-\left(\vw^T_i \widetilde{\vw}_i\right)^2} \leq \sum_{i=1}^r \sqrt{1-\cos^2(\theta_r)} = r\sin(\theta_r).
\end{equation}
\end{proof}

Now, applying the Cauchy-Schwarz inequality to $\left(\vx^T (\widetilde{\mW} - \mW) \nabla_\vu g(\mW^T \vx)\right)^2$ in \eqref{eq:mse} yields the following:
\begin{align}
\begin{split} \label{equ:bound1}
\epsilon &\leq \myEx\left[\left(\nabla_\vu g^T \nabla_\vu g \right) \left(\vx^T (\widetilde{\mW} - \mW)(\widetilde{\mW} - \mW)^T \vx\right)\right]\\ 
&\leq G^2 \underbrace{\myEx\left[\vx^T (\widetilde{\mW} - \mW)(\widetilde{\mW} - \mW)^T \vx\right]}_{:= \eta}.
\end{split}
\end{align}
Let $\mA = \widetilde{\mW} - \mW$, $\vy = \mA^T \vx$ and $\va_i$ be the $i$-th column of $\mA$. Then,
\begin{equation} \label{equ:eta}
\eta = \myEx\left[\vy^T \vy \right] = \sum_{i=1}^r \va_i^T \myEx\left[\vx \vx^T\right] \va_i = \sigma_x^2 \sum_{i=1}^r \va_i^T \va_i.
\end{equation}
However, we have $\va_i = \widetilde{\vw}_i - \vw_i$. So,
\begin{align}
\begin{split} \label{equ:ata}
\va_i^T \va_i &= (\widetilde{\vw}_i - \vw_i)^T (\widetilde{\vw}_i - \vw_i)\\
&= 2-2\widetilde{\vw}_i^T \vw_i,
\end{split}
\end{align}
where we have used the fact that $\widetilde{\vw}_i^T \widetilde{\vw}_i = \vw_i^T \vw_i = 1$. So, substituting \eqref{equ:eta} and \eqref{equ:ata} into \eqref{equ:bound1}, we have:
\begin{equation} \label{equ:full_bound}
\epsilon \leq G^2 \sigma_x^2 \sum_{i=1}^r \left(2-2\widetilde{\vw}_i^T \vw_i\right).
\end{equation}
Applying the sufficient part of \Cref{lem:inner_prod} completes the proof.

\section{Global linear models} \label{sec:lin_models} The coefficients of a global linear model of a qoi give a rough estimation of the leading dimension-reducing subspace direction. If the qoi varies largely linearly in the input domain, and it is known that the dimension-reducing subspace is one-dimensional, this heuristic will find the required subspace at a low cost. This method is described with Algorithm 1.3 in \cite{constantine2015active}, which we reproduce below.
\begin{enumerate}
\item Draw input/output pairs $\left(\vx^{(m)}, f\left(\vx^{(m)}\right)\right)_{m=1}^M$ where $M = \alpha d$ with a certain oversampling $\alpha > 1$.
\item Solve the following least squares problem
\begin{equation}
\text{minimise}_{c,\vw} \quad \norm{\mX \vw + c - \vy}{2},
\end{equation}
where the $m$-th row of $\mX$ is $\vx^{(m)}$ and $y_m = f(\vx^{(m)})$.
\item Take the leading ridge direction to be $\vw / \norm{\vw}{2}$.
\end{enumerate}

\section{MAVE} \label{sec:MAVE} The Minimum Average Variance Estimation (MAVE) method for sufficient dimension reduction is proposed by Xia et al. \cite{xia2002adaptive}. Given input/output pairs $(\vx^{(m)}, y^{(m)})_{m=1}^M$, it aims to find the dimension-reducing subspace spanned by columns of $\mW$, which is the solution to the following optimisation over matrices with orthogonal columns:
\begin{align} \label{eqn:MAVE}
\begin{split}
\text{minimise}_{\mW} \qquad &\mathbb{E}\left[\left(y - \mathbb{E}\left[y \, | \, \mW^T \vx\right]\right)^2\right]\\
\text{subject to} \qquad &\mW^T \mW = \mI.
\end{split}
\end{align}
The first-order Taylor expansion about an input $\vx^{(0)}$ of the expectation within the parentheses is
\begin{equation}
\mathbb{E}\left[y^{(i)} \, | \, \mW^T \vx^{(i)}\right] \approx a_0 + \vb_0^T \mW^T(\vx^{(i)} - \vx^{(0)}).
\end{equation}
Hence, the MAVE procedure minimises the following approximation to \eqref{eqn:MAVE} as an alternating weighted least squares problem:
\begin{align} \label{eqn:alt_MAVE}
\begin{split}
\text{minimise}_{a_j, \vb_j, \mW} \qquad &\sum_{i=1}^M \sum_{j=1}^M \left[y^{(i)} - a_j - \vb_j^T \mW^T(\vx^{(i)} - \vx^{(j)})\right]^2 \omega_{ij}\\
\text{subject to} \qquad &\mW^T \mW = \mI,
\end{split}
\end{align}
where the weights $\omega_{ij}$ are determined through a normalised kernel function $K_h$ evaluated at $(\vx_i, \vx_j)$, namely
\begin{equation}
\omega_{ij} = \frac{K_h\left(\mW^T(\vx^{(i)} - \vx^{(j)})\right)}{\sum_{k=1}^M K_h\left(\mW^T(\vx^{(k)} - \vx^{(j)})\right)}.
\end{equation}
The procedure iterates with
\begin{itemize}
\item fixing $a_j, \vb_j$ and optimizing with respect to $\mW$, and
\item fixing $\mW$ and optimizing with respect to $a_j, \vb_j$.
\end{itemize}
The minimiser for both steps can be expressed analytically, with details in, e.g., \cite[Ch.~11]{li2018sufficient}.

\bibliographystyle{elsarticle-num}
\bibliography{paper}

\end{document}

%% file: pauls-macros.tex
\usepackage{bm}

\graphicspath{{./}{./figs/}}

  \def\clap#1{\hbox to 0pt{\hss#1\hss}}

\providecommand{\mat}[1]{\bm{#1}}%
\renewcommand{\vec}[1]{\mathbf{#1}}

\providecommand{\mA}{\ensuremath{\mat{A}}}

\providecommand{\mC}{\ensuremath{\mat{C}}}

\providecommand{\mE}{\ensuremath{\mat{E}}}

\providecommand{\mH}{\ensuremath{\mat{H}}}
\providecommand{\mI}{\ensuremath{\mat{I}}}
\providecommand{\mJ}{\ensuremath{\mat{J}}}

\providecommand{\mR}{\ensuremath{\mat{R}}}

\providecommand{\mU}{\ensuremath{\mat{U}}}

\providecommand{\mW}{\ensuremath{\mat{W}}}
\providecommand{\mX}{\ensuremath{\mat{X}}}

\providecommand{\va}{\ensuremath{\vec{a}}}
\providecommand{\vb}{\ensuremath{\vec{b}}}

\providecommand{\vf}{\ensuremath{\vec{f}}}

\providecommand{\vn}{\ensuremath{\vec{n}}}

\providecommand{\vs}{\ensuremath{\vec{s}}}

\providecommand{\vu}{\ensuremath{\vec{u}}}

\providecommand{\vw}{\ensuremath{\vec{w}}}
\providecommand{\vx}{\ensuremath{\vec{x}}}
\providecommand{\vy}{\ensuremath{\vec{y}}}



%% file: sec-1.tex
\section{Introduction} \label{sec:intro}
The governing physics in many engineering problems is described by a system of partial differential equations (PDEs). These equations can be solved by suitable discretisation methods such as finite element and finite volume methods, where scalar fields---e.g.~pressures, temperatures, strains---are computed at each node over the PDE domain. One can interpret these scalar fields as vector-valued functions, conditioned upon certain boundary conditions and geometry parameters. Here each value of the output vector corresponds to the scalar field quantity at a specific node of the domain. Integrals of these scalar field variables typically constitute output qois in uncertainty quantification studies. For instance, when propagating uncertainties in the Mach number and angle of attack of flow over an airfoil, one is interested in quantifying the moments of the lift and drag coefficients \cite{witteveen2009uncertainty}; both lift and drag coefficients are surface integrals of the static pressure field around the airfoil \cite[Ch.~1]{anderson2010fundamentals}. In studying the impact of uncertainties in leakage flows in a compressor \cite{seshadri2015leakage}, one is interested in quantifying the moments of isentropic efficiency, which can be expressed as an integral of the pressure and temperature ratios \cite[p.~7]{denton1993loss}. Qois that are integrals of such scalar fields are prevalent well beyond computational fluid dynamics (CFD). For instance, the displacement of a structure is the integral of the strain field \cite[sec.~2.1]{sun2006mechanics}, which allows us to analyse linear elastic displacement problems in areas including soil mechanics \cite[sec.~6.4]{blatman2011adaptive} and machine component design \cite[sec.~5.2]{lam2020multifidelity} using the strain field.

Uncertainty quantification studies typically require a design of experiment, where the governing PDE model is evaluated under different inputs.  The number of times the model is evaluated depends on the dimension of the input space and the non-linearity of the scalar qois. In this paper, we explore a deviation from this paradigm. Rather than storing scalar qois for each model evaluation, we explore whether one can reduce the number of model evaluations if one stores select scalar fields. More specifically, we want to design emulators for the scalar fields themselves and then integrate the emulators to obtain the desired qois. 

So, why should we focus on scalar fields? Consider the following examples. Over an airfoil in subsonic flow, the static pressure of a point near the leading edge is unlikely to be strongly affected by small geometric perturbations far downstream. In a similar vein, the local deflection of a structure is unlikely to be affected by small local changes in elastic properties far away from the point of measurement. We refer to this property of physical scalar fields as \emph{localisation}---the output at each node is only influenced by perturbations local to the node. Mathematically, a parameterised scalar field can be expressed as a function $f(\vx, \vs)$, where $\vx$ denotes the perturbation parameters and $\vs$ the spatial location. We say that the field is localised if for each $\vs$, the scalar field depends mainly on a smaller subset of $\vx$ localised within a region\footnote{Localisation here refers to the fact that the physical perturbations influencing each node are contained spatially. This is in contrast with physical features that are localised spatially, such as shock waves and phase boundaries. The latter refer to a localised set of \emph{output nodes}, but the dependence of these nodal values on the input need not be localised.}. Provided that this smaller set of variables can be found, the \emph{curse of dimensionality} can be abated and the number of experiments required for approximating the field can be greatly reduced. A related scalar field property is \emph{smoothness}, which implies a certain degree of continuity in the variation of the field and its derivatives. Smoothness can be defined with respect to the parameters $\vx$ and/or the spatial location $\vs$. In this paper, we assume smoothness with respect to $\vx$, such that smooth functions can be used to form approximation models. Furthermore, we show that smoothness in $\vs$ can be leveraged to form compressed representations of scalar fields, facilitating the reconstruction of scalar fields from approximation.

To construct low-dimensional approximations of fields, we draw ideas from ridge functions \cite{pinkus2015ridge}. A function $f:\myRe^d \rightarrow \myRe$ whose variation is entirely contained within a subspace described by ran($\mW$), where $\mW\in \myRe^{d \times r}$ has orthonormal columns and $\text{ran}(\cdot)$ denotes the column space, is called a \emph{generalised ridge function} \cite{pinkus2015ridge}. That is, it can be expressed as
\begin{equation}
f(\vx) = g\left(\mW^T \vx \right).
\end{equation}
In this paper, we will refer to these functions simply as \emph{ridge functions} for brevity. Many physical qois are characterised by anisotropy in the input domain---i.e.~they vary strongly only within a subspace---such as the localised output components described in the previous paragraph. These qois can be well-approximated by ridge functions, and the process of finding these ridge functions is known as \emph{ridge approximation}. Namely, we find $\mW\in \myRe^{d \times r}$ with orthonormal columns and $g: \myRe^r \rightarrow \myRe$ such that
\begin{equation}
f(\vx) \approx g\left(\mW^T \vx \right),
\end{equation}
where the approximation can be formulated by minimising the mean squared error (MSE) over the input domain \cite[sec.~3.1]{constantine2014active}. When given a ridge approximation, significant computational run-time savings can be achieved by working with $g\left(\mW^T \vx \right)$ as an emulator for $f(\vx)$ if $r \ll d$, where we have effectively reduced the dimension of our problem from $d$ to $r$. We refer to the span of the columns of $\mW$ as \emph{ridge directions}, and $g$ as the \emph{ridge profile}. As $\mW$ takes the $d$-dimensional input to an $r$-dimensional projection, we also call the column space of $\mW$ the \emph{dimension-reducing subspace}. 

Numerous methods for ridge approximation have been proposed in the literature. Central to this paper are strategies based on analysis of the average outer product of the gradient with itself, which we will call the \emph{gradient covariance matrix} of $f$. Assuming that $f$ is Lipschitz continuous with bounded first derivatives, this matrix is defined as
\begin{equation} \label{equ:scalar_covariance}
\mC(f) = \int_{\mathcal{D}} \nabla f(\vx) \nabla f(\vx)^T \rho(\vx) d\vx = \myEx \left[\nabla f(\vx) \nabla f(\vx)^T \right] ,
\end{equation}
where $\mathcal{D}$ is the input domain, and $\rho: \myRe^d \rightarrow \myRe$ is a function that integrates to unity and is strictly positive within $\mathcal{D}$. We assume that all entries in this matrix are finite. Interpreting $\rho$ as a probability density function (PDF) over the input domain, we can replace the integral by the expectation over $\vx$, which we treat as a random variable\footnote{In the following, unless specified otherwise, all expectations are computed with respect to $\rho(\vx)$.}. Samarov \cite{samarov1993exploring} and Constantine et al.~\cite{constantine2015active,constantine2014active} analyse the eigendecomposition of this matrix and show that, if the variation of the function outside the span of eigenvectors corresponding to large eigenvalues is small, then a suitable choice for the ridge directions is the leading eigenvectors of the gradient covariance matrix. The subspace spanned by these leading eigenvectors is termed the \emph{active subspace} of $f$ by Constantine et al.~\cite{constantine2014active}. In a recent paper by Zahm and co-authors \cite{zahm2020gradient-based}, an extension to vector-valued functions is proposed. In \Cref{sec:Embedded_ridge_functions} we will explore the properties of their  \emph{vector gradient covariance matrix}. 

In practice, gradients of the qoi are often difficult to evaluate; for example, computational models may be based on legacy codes without automatic differentiation capabilities. In the absence of gradient information, other methods of ridge approximation have been proposed. Fornasier et al.~\cite{fornasier2012learning} and Tyagi and Cevher~\cite{tyagi2014learning} describe methods to recover low-dimensional ridge structures with finite differencing through compressed sensing and low-rank recovery, respectively. Hokanson and Constantine~\cite{hokanson2018data-driven} describe an algorithm to form a polynomial ridge approximation from fewer data than required for a full response surface by solving a non-linear least squares problem with \emph{variable projection} (VP) \cite{golub2003separable}. Constantine et al.~\cite{constantine2015computing} and Eftekhari et al.~\cite{eftekhari2017learning} describe methods to estimate the ridge directions based on finite differences. In addition, Glaws et al.~\cite{glaws2020inverse} draw an analogy between ridge recovery and \emph{sufficient dimension reduction} (SDR). In SDR, the goal is to find the minimal subspace, described by the column span of a matrix $\mW$, with which we can establish conditional independence between a set of covariates $\vx$ and the response $y$ in a regression setting. That is, we seek a subspace described by $\mW$ such that $y\perp \vx$ given $\mW^T \vx$. This \emph{central subspace} is intimately linked to the ridge directions \cite[thm.~2]{glaws2020inverse}, which implies that regression-based methods within the study of SDR can be applied to ridge approximation (see Section 2.2 in \cite{seshadri2019dimension} for further discussion). Examples of such methods include Sliced Inverse Regression (SIR) \cite{li1991sliced}, Sliced Average Variance Estimation (SAVE) \cite{cook1991sliced} and Minimum Average Variance Estimation (MAVE) \cite{xia2002adaptive}.  The former two techniques are based on inverse regression and the latter based on forward regression; given a set of predictor/response pairs $\{\vx_i, y_i\}_{i=1}^M$, forward regression aims to estimate statistics of the distribution of the response given covariates ($y|\vx$), while inverse regression aims to characterise $\vx|y$.

Given the ridge directions $\mW$, the ridge profile $g$ that minimises the mean squared approximation error $\myEx\left[\left(f(\vx) - \widetilde{g}\left(\mW^T \vx \right)\right)^2\right]$ over $\widetilde{g}$ can be shown analytically to be \cite{constantine2014active,zahm2020gradient-based}
\begin{equation} \label{equ:ridge_profile}
g\left(\mW^T \vx \right) = \myEx\left[f(\vx)\,|\,\mW^T \vx\right].
\end{equation}
Several approaches for approximating this average have been proposed, including the use of Gaussian processes \cite{seshadri2019dimension} and multivariate orthogonal polynomials \cite{glaws2019gausschristoffel}.

In this paper, we build upon these ideas and introduce an \emph{embedded ridge approximation} to connect the ridge approximation of desired qois with ridge approximations of their constituent scalar fields. Loosely stated, an embedded ridge approximation is based on vector-valued functions, where each component can be approximated by a ridge function---computed for scalar fields at select nodes within the computational domain. Leveraging gradient-free techniques, a surrogate model for the scalar field is constructed via ridge approximations at each node. When provided with certain structural assumptions about the field---including localisation---we show that it is possible to reduce the number of model evaluations for computing the dimension-reducing subspace of a related qoi, by exploiting these nodal ridge approximations.

The rest of the paper is structured as follows: in \Cref{sec:Embedded_ridge_functions}, we describe the concept of an \emph{embedded ridge function}, and provide an algorithm supported by theoretical analysis that leverages this embedded structure to find the dimension-reducing subspace of a scalar qoi based on an underlying spatial field. In \Cref{sec:storage}, we leverage the similarity between neighbouring output components to motivate a method of efficiently storing the array of ridge directions associated with each output component, and provide some algorithms for this process. In \Cref{sec:examples}, we use analytical and numerical examples to illustrate the algorithms proposed in this paper. 

\emph{Notation.} We denote the approximation of a quantity with a hat. For instance, a finite-sample estimate of a matrix $\mH$ is denoted as $\widehat{\mH}$ and a response surface for $g$ fitted with finitely many samples is denoted $\widehat{g}$. A general perturbation of a quantity is denoted with a tilde. For instance, a perturbation of $\vw$ is denoted $\widetilde{\vw}$.

%% file: sec-2.tex
 \section{Embedded ridge approximation} \label{sec:Embedded_ridge_functions}
Consider the scalar field $f(\vx,\vs)$ where $\vx \in \mathcal{D} \subset \myRe^d$ parameterises our model of interest and $\vs \in \myRe^K$ is a variable that denotes the spatial location in $K$-dimensional space. We place the following assumptions on this field:
\begin{enumerate}
\item the input space $\mathcal{D}$ is endowed with a probability density $\rho(\vx)$;
\item the field is square integrable with respect to the probability density $\rho(\vx)$ and Lipschitz continuous with bounded and square integrable first partial derivatives with respect to $\vx$ for all $\vs$.
\end{enumerate}
An example of $f(\vx,\vs)$ is the pressure within a computational domain---characterised by geometry parameters or boundary conditions $\vx$ and a probe location $\vs$. Our goal is to study dimension-reducing subspaces induced by the function
\begin{equation}
h(\vx) = \int_\mathcal{D} \omega(\vs) f(\vx,\vs)\; \text{d}\vs,
\end{equation}
where $\omega: \myRe^K \rightarrow \myRe$ is a weight function. In other words, $h(\vx)$ represents a weighted average of the scalar field $f(\vx,\vs)$, and is the relevant qoi. If $f(\vx,\vs)$ is the pressure distribution, then one can think of $h(\vx)$ as the lift coefficient or drag coefficient, for instance. Assuming that the partial derivative of the integrand is bounded independent of $\vx$ and $\vs$, we can write
\begin{equation}
\nabla_{\vx} h(\vx) = \int_\mathcal{D} \omega(\vs) \nabla_\vx f(\vx,\vs) \; \text{d}\vs.
\end{equation}
Now, let us assume that $f(\vx,\vs)$ can be approximated by a ridge function of the form
\begin{equation} \label{equ:comp_ridge_approx}
f(\vx,\vs) \approx g_{\vs} \left(\mW_{\vs}^T \vx\right),
\end{equation}
where $\mW_{\vs} \in \myRe^{d\times r_{\vs}}$, where $r_{\vs}$ identifies the dimension of the subspace spanned by the orthonormal columns of $\mW_{\vs}$. Note that $g_{\vs}$, $\mW_{\vs}$ and $r_{\vs}$ will, in general, depend on $\vs$. The gradient of $f$ is then given by
\begin{equation}
\nabla_{\vx} f(\vx,\vs) \approx \frac{\partial g_{\vs}\left(\mW_{\vs}^T \vx\right)}{\partial \vx} = \mW_{\vs} \nabla g_{\vs}\left(\mW_{\vs}^T \vx\right),
\label{equ:grad_f}
\end{equation}
noting that $\nabla g_{\vs}\left(\mW_{\vs}^T \vx\right) := \text{d}g_{\vs}(\vu)/\text{d}\vu|_{\vu = \mW_{\vs}^T \vx}$. Thus,
\begin{equation}
\nabla_{\vx} h(\vx) \approx \int_\mathcal{D} \omega({\vs}) \mW_{\vs} \nabla g_{\vs}\left(\mW_{\vs}^T \vx\right) \text{d} {\vs}.
\end{equation}
In practice, we can approximate this gradient with an $N$-point quadrature rule $\left\{ {\vs}_{i},\omega_{i}\right\} _{i=1}^{N}$, with quadrature points ${\vs}_i$ and quadrature weights $\omega_i$, from which we arrive at the approximation
\begin{align} \label{eqn:approx_gradients}
\begin{split}
\nabla_{\vx} h(\vx) & \approx \sum_{i=1}^N \omega_i \mW_{{\vs}_i} \nabla g_{{\vs}_i} \left(\mW_{{\vs}_i}^T \vx\right), \\
& = \sum_{i=1}^N \omega_i \mW_i \nabla g_i\left(\mW_i^T \vx\right),
\end{split}
\end{align}
where we have expressed $g_{{\vs}_i}$ as $g_i$ and $\mW_{{\vs}_i}$ as $\mW_{i}$ for notational convenience. Then, from \eqref{equ:scalar_covariance} we can approximate the scalar covariance matrix of $h(\vx)$ as
\begin{align} \label{equ:Ch}
\begin{split}
\mC(h) &:= \myEx\left[\nabla_{\vx} h (\vx)\nabla_{\vx} h(\vx)^T \right]\\
&\approx \myEx\left[\left(\sum_{i=1}^N \omega_i \mW_i \nabla g_i\left(\mW_i^T \vx\right)\right) \left(\sum_{j=1}^N \omega_j \mW_j \nabla g_j\left(\mW_j^T \vx\right)\right)^T\right]\\
&= \sum_{i=1}^N \sum_{j=1}^N \omega_i \omega_j \mW_i \myEx\left[\nabla g_i \nabla g_j^T\right] \mW_j^T.
\end{split}
\end{align}
The fact that \eqref{equ:Ch} is expressed in terms of the ridge parameters $\mW_i$ and $g_i$ is noteworthy. Given $\mC(h)$, we can find the dimension-reducing subspace of $h$ simply via an eigendecomposition, and this equation informs us that it is possible to shift the computational burden from evaluating the gradients $\nabla_\vx h$ to the estimation of the ridge parameters $\mW_i$ and $g_i$. In the absence of automatic differentiation or adjoint solvers, the former may require finite differences or the use of a surrogate model, whose cost of formulation suffers from the curse of dimensionality. Supposing that the scalar field $f(\vx, {\vs})$ is localised (see \Cref{sec:intro}), at a fixed location ${\vs}_i$ the dependence of $f(\vx,{\vs}_i)$ on $\vx$ is likely to be highly anisotropic, depending mainly on the parameters that impact nodes adjacent to ${\vs}_i$. This implies that the number of ridge directions at each location is likely to be small, and $g_i$ will be a low-dimensional function. For many methods of ridge approximation, this implies that the amount of simulation data required for this step is reduced for a given approximation accuracy. This brings us to the central idea of this paper: instead of directly calculating the dimension-reducing subspace of $h$, we leverage the decomposition of $h$ into its composite scalar field, whose dimension-reducing subspaces are likely to be inexpensive to compute. Using evaluations of the field, gradient-free strategies for ridge approximations---such as VP and MAVE, mentioned in \Cref{sec:intro}---can be used to furnish ridge approximations at each node of the scalar field. Then, from these component subspaces, we assemble the scalar gradient covariance of $h$, where the required gradients are calculated using approximations formed at the nodes via \eqref{eqn:approx_gradients}.

\subsection{Interpretation as vector-valued dimension reduction}
We assume that there exists a set of \emph{quadrature points} ${\vs}_i$ (domain nodes) to evaluate the scalar field $f(\vx,{\vs})$ that allows us to formulate the ridge approximation \eqref{equ:comp_ridge_approx} easily. In practice, this set of points is fixed by the computational domain and its associated mesh. We argue that reducing the dimensionality of the qoi $h$ is facilitated by the formulation \eqref{equ:Ch}. We can treat the scalar field evaluated at the prescribed positions as a vector-valued function $\vf(\vx)$, whose components $f_i(\vx) := f(\vx,{\vs}_i)$ exhibit ridge approximations. 
\begin{definition}[Embedded ridge function]
Let $\vf$ be a vector-valued function $\vf: \myRe^d \rightarrow \myRe^N$ with components $f_1(\vx), f_2(\vx),..., f_N(\vx)$, where each $f_i:\myRe^d \rightarrow \myRe$. Such a function $\vf$ is called an \emph{embedded ridge function} if it satisfies
\begin{equation} \label{embedded_ridge_defn}
\mb{f}(\mb{x}) = \begin{bmatrix}
f_1 (\mb{x})\\
\vdots \\
f_N (\mb{x})
\end{bmatrix}
=
\begin{bmatrix}
g_1 \left(\mW_1^T\mb{x}\right)\\
\vdots \\
g_N \left(\mW_N^T\mb{x}\right)
\end{bmatrix},
\end{equation}
where the $i$-th element of the vector, $f_{i} \left( \vx \right)$, is a ridge function of the form $g_{i} \left( \mW_{i}^{T} \vx \right)$, for $i=1, \ldots, N$.  Here all subspace matrices $\mW_{i}$ have the same number of rows $d$, but may have different numbers of columns $r_i$. It is assumed that the components of $\vx$ are independent under the input measure $\rho$. 
\end{definition}
An approximation of a vector-valued function using embedded ridge functions of the form \eqref{embedded_ridge_defn} is called an \emph{embedded ridge approximation}. There are parallels between an embedded ridge approximation and vector-valued dimension reduction. In \cite{zahm2020gradient-based}, the authors introduce a vector gradient covariance matrix, analogous to the scalar form given in \eqref{equ:scalar_covariance}.
\begin{definition}[Vector gradient covariance matrix] 
We define the \emph{vector gradient covariance matrix} $\mH(\vf) \in \myRe^{d\times d}$ as
\begin{equation} \label{vector_covariance}
\mH(\vf) = \myEx\left[\mJ (\vx) \mR \mJ(\mb{x})^T\right] = \int_{\mathcal{D}} \mJ (\vx) \mR \mJ(\mb{x})^T \rho(\mb{x}) d\mb{x},
\end{equation}
where $\mR\in \myRe^{N\times N}$ is a symmetric positive semi-definite matrix of weights, and
\begin{equation} \label{jacobian}
\mJ \left( \vx \right) = \left[ \frac{\partial f_1}{\partial \vx} , \ldots , \frac{\partial f_N}{\partial \vx} \right],
\end{equation}
where $\mJ \in \mathbb{R}^{d \times N}$ is the Jacobian matrix. 
\end{definition}
Observe that for an embedded ridge function, by setting 
\begin{equation}
\mR = \boldsymbol{\omega}\boldsymbol{\omega}^T, \; \; \; \text{where} \; \; \; \boldsymbol{\omega} = \left( \omega_1,\omega_2, \ldots ,\omega_N \right)^{T},
\end{equation}
we can recover the final line of \eqref{equ:Ch}. Thus, the embedded ridge approximation can be constructed by calculating the vector gradient covariance matrix of the discretised weighted scalar field. Note that, although this vector gradient covariance matrix coincides with the formulation in \cite{zahm2020gradient-based}, our focus is on a weighted average of the underlying field instead of the vector-valued function in itself. 
\subsection{Algorithm for ridge computation}
Equipped with a localised scalar field, we can design a procedure for the embedded ridge approximation of a qoi. Here our ridge function has the form
\begin{equation}
h(\vx) = \int_\mathcal{D} \omega(s) f(\vx,{\vs})\; \text{d}{\vs} \approx \boldsymbol{\omega}^T \vf(\vx).
\end{equation}
In \Cref{alg:embedded_ridge_comp}, we identify $N$ sets of ridge directions---one for every component of the vector $\vf$---assuming the absence of gradient information. We then fit ridge profiles $g_i$ to obtain nodal ridge approximations that we use for computing the scalar gradient covariance matrix for $h$.
\begin{algorithm}[h]
\caption{Embedded ridge function approximation.}
\label{alg:embedded_ridge_comp}
\begin{algorithm2e}[H]
\KwData{Input/output pairs $\left(\vx^{(m)}, \vf^{(m)}\right)_{m=1}^M$ with $\vf^{(m)} = \left[f_1\left(\vx^{(m)}\right),f_2\left(\vx^{(m)}\right),...,f_N\left(\vx^{(m)}\right)\right]^T$.}
\KwResult{Ridge profile $\widehat{g}_h$ and ridge directions $\widehat{\mU}$ such that $h(\vx) \approx \widehat{g}_h\left(\widehat{\mU}^T \vx\right)$ }
\For {$i = 1,...,N$}
{
Find $\widehat{\mW}_i$ with orthonormal columns using a gradient-free ridge approximation strategy.\footnote{Note that we can use the same set of input values for each component.}\\
Fit an approximate ridge profile $\widehat{g}_i$ using 
\begin{equation}
\left\{\left(\widehat{\mW}_i^T \vx^{(1)}, f_{i}^{(1)}\right), \ldots, \left(\widehat{\mW}_i^T \vx^{(M)}, f_{i}^{(M)}\right)\right\}
\end{equation}
as training data.\\
Evaluate $\nabla_{\vx} \widehat{f_i}\left(\vx^{(m)}\right)$ with \eqref{equ:grad_f} for $m = 1,...,M$.\\
}
Form $\widehat{\mJ}\left(\vx^{(m)}\right)$ from \eqref{jacobian}.\\
Calculate
\begin{equation} \label{equ:MC_Ch}
\widehat{\mC}(h) = \frac{1}{M} \sum_{m=1}^M \widehat{\mJ}\left(\vx^{(m)}\right) \mR \widehat{\mJ}\left(\vx^{(m)}\right)^T
\end{equation}
where 
\begin{equation}
\mR = \boldsymbol{\omega}\boldsymbol{\omega}^T.
\end{equation}
\\
Find the eigendecomposition of $\widehat{\mC}$ and choose the leading eigenvectors with the largest eigenvalues to form $\widehat{\mU}$.\\
Fit a low-dimensional ridge approximation $\widehat{g}_h$ using 
\begin{equation}
\left\{\left(\widehat{\mU}^T \vx^{(1)}, \boldsymbol{\omega}^T \vf^{(1)}\right), \ldots, \left(\widehat{\mU}^T \vx^{(M)}, \boldsymbol{\omega}^T \vf^{(M)}\right)\right\}
\end{equation}
as training data.
\end{algorithm2e}
\end{algorithm}

In the process of embedded ridge approximation, a ridge approximation for each node of the field is computed, exposing the rich structure endowed by the localisation of the field, thus enabling a reduction in the number of required training samples. These approximations can be used to form surrogate models of any qoi derived from the field, by replacing field evaluations by evaluations of the embedded ridges. However, we note that, in many situations, it is valuable to deduce a dimension-reducing subspace for the derived qoi, because the subspace facilitates many more tasks beyond surrogate modelling. Examples include optimisation \cite{gross2020optimisation,lukaczyk2014active}, visualisation \cite{cook1998regression} (if the subspace is one- or two-dimensional), sensitivity analysis \cite{wong2019extremum} and discovery of physical insights \cite{del_rosario2018dimension}. The ability to visualise the variation of the qoi is especially important in the design process to easily gauge whether an approximation model can be trusted. Thus, the subspace computed from step 7 onwards in \Cref{alg:embedded_ridge_comp} plays an important role in the embedded ridge approximation approach.

A localised scalar field contains nodes that are well-approximated with low-dimensional ridge subspaces. These subspaces are usually of lower dimensionality than the ridge subspace of derived qois, implying that each $\mW_i$ usually has fewer columns than $\mU$ in \Cref{alg:embedded_ridge_comp}. This gives the embedded ridge approximation approach an advantage---since low-dimensional ridge functions can be synthesised into a less easily found surrogate for qois. In what follows, we quantify this notion by studying the error on the gradient covariance matrix $\mC(h)$ via the estimate $\widehat{\mC}(h)$ computed in \eqref{equ:MC_Ch}. We establish a bound on the expected norm difference $\myEx\norm{\mC(h) - \widehat{\mC}(h)}{2}$ with the matrix Bernstein inequality \cite{tropp2012user-friendly}, where the matrix norm $\norm{\cdot}{2}$ yields the largest singular value of the argument. The accuracy of nodal ridge approximation is modelled  with a quantity $\eta_M$ dependent on $M$.\footnote{A better measure of this error can be defined using the more general subspace distance \eqref{equ:subspace_distance}, which is basis-agnostic. However, it can be shown, in a similar manner to \Cref{lem:inner_prod}, that a basis can be found such that a small subspace distance is equivalent to the present bound in this section.} For ease in exposition, we also assume that the component ridge dimension $r$ is constant in space, so $\mW_i \in \myRe^{d\times r}$ for all $i$.
\begin{theorem} \label{thm:sample_bound}
Assume that $\norm{\nabla g_i}{2} \leq L$ for all $1\leq i \leq N$,
\begin{equation} \label{equ:M_cond}
M \geq \frac{2L^2 \log(2r)}{\epsilon^2 \norm{\myEx\left[\nabla g_i \nabla g_j^T\right]}{2}},
\end{equation}
for all $1\leq i,j \leq N$, and 
\begin{equation*}
\norm{\widehat{\mW}_i - \mW_i}{2} \leq \eta_M.
\end{equation*}
Then,
\begin{equation*}
\myEx\norm{\widehat{\mC}(h) - \mC(h)}{2} \leq L^2 C \left(2\eta_M + \epsilon + \epsilon^2\right),
\end{equation*}
where $C := \sum_{ij} |\omega_i \omega_j|$.
\end{theorem}
\begin{proof}
See \Cref{sec:sample_bound_proof}.
\end{proof}

There are several important remarks to make regarding \eqref{equ:M_cond}. It should be clear that the number of samples $M$ required scales as a function of $r$ instead of $d$. This encapsulates the advantage brought about by considering the locality of the scalar field. In \Cref{sec:examples} we provide numerical studies that illustrate this. It is possible to derive a bound related to the subspace error of $\widehat{\mU}$ in \Cref{alg:embedded_ridge_comp} as a corollary to \Cref{thm:sample_bound}. This involves steps very similar to Lemma 3.9 and Corollary 3.10 in \cite{constantine2015active}, which are based on Corollary 8.1.11 in \cite{golub2013matrix}.

%% file: sec-3.tex
\section{Efficient storage of embedded ridge approximations: Ridge compression} \label{sec:storage}
Scalar field quantities are propagated through a PDE domain node-by-node. It is therefore very likely that neighbouring nodes---depending on the overall resolution of the mesh---will have similar values of these quantities. More specifically, one can think of these quantities as being strongly correlated with their neighbours---i.e.~the underlying scalar field is smooth \emph{spatially}. When approximating each component $f_i$ as a ridge function, this correlation can be interpreted as similarity in both the ridge directions $\mW_i$ and the ridge profile $g_i$. In this section, we assume that the number of ridge directions for each $\mW_i$ is constant, and equal to $r$, for simplicity.

\subsection{A perturbation bound on the mean squared error} \label{sec:storage_bound}
Given a ridge function $g(\mW^T \vx)$, consider a perturbation of the dimension-reducing subspace from $\mathcal{S} = \text{ran}(\mW)$ to $\widetilde{\mathcal{S}}$. This perturbation is quantified by the \emph{subspace distance}, defined as
\begin{equation} \label{equ:subspace_distance}
dist(\mathcal{S}_1, \mathcal{S}_2) = \norm{\mW_1\mW_1^T - \mW_2\mW_2^T}{2},
\end{equation}
for two subspaces $\mathcal{S}_1$ and $\mathcal{S}_2$ of $\myRe^{d}$. Here, $\mW_1$ and $\mW_2$ are two matrices with orthonormal columns such that $\mathcal{S}_1 = \text{ran}(\mW_1)$ and $\mathcal{S}_2 = \text{ran}(\mW_2)$ respectively. The goal is to characterise the error incurred by the perturbation in the dimension-reducing subspace. For a basis matrix $\widetilde{\mW}$ of $\widetilde{\mathcal{S}}$, we can form a Taylor expansion
\begin{equation}
g\left(\widetilde{\mW}^T \vx \right) = g\left(\mW^T \vx\right) + \left((\widetilde{\mW} - \mW)^T \vx\right)^T \underbrace{\nabla g(\vu)|_{\vu = \mW^T \vx}}_{\nabla_\vu g(\mW^T \vx)}~+~\text{h.o.t.}
\end{equation}
If the subspace perturbation is small enough, higher-order terms (h.o.t.) can be neglected and the mean squared error can be approximated as
\begin{equation} \label{eq:mse}
\myEx\left[\left(g\left(\widetilde{\mW}^T \vx \right) - g\left(\mW^T \vx\right) \right)^2\right] \approx \epsilon =  \myEx \left[\left(\vx^T (\widetilde{\mW} - \mW) \nabla_\vu g\left(\mW^T \vx \right)\right)^2\right].
\end{equation}
Clearly, the quantity $\epsilon$ depends on the specification of basis matrices, which are not fixed for given subspaces. In the following theorem, we show that it is possible to select basis matrices that allow $\epsilon$ to be bounded by a function of the perturbation  distance.
\begin{theorem} \label{thm:subspace_perturb}
Let $\mathcal{S} = \text{ran}(\mW)$, and $\widetilde{\mathcal{S}}$ be a perturbation of $\mathcal{S}$. Assume that the square of the gradient is bounded as $\nabla_\vu g^T \nabla_\vu g\leq G^2$ and $\myEx[\vx\vx^T] = \sigma_x^2 \mI_d$ (i.e.~inputs are independent and identically distributed). Then, if $dist(\mathcal{S}, \widetilde{\mathcal{S}}) \leq \sin(\theta_r)$, we can pick $\mW, \widetilde{\mW} \in \myRe^{d\times r}$ where $\mathcal{S} = \text{ran}(\mW)$ and $\widetilde{\mathcal{S}} = \text{ran}(\widetilde{\mW})$ such that 
\begin{equation}
\epsilon \leq G^2 \sigma_x^2 \sum_{i=1}^r (2-2\cos(\theta_r)),
\end{equation}
where $\epsilon$ is the first-order approximation to the mean squared error \eqref{eq:mse}.
\end{theorem} 
\begin{proof}
See \Cref{sec:subspace_proof}.
\end{proof}

\Cref{thm:subspace_perturb} establishes a stability bound on the approximation error of a ridge function with a small perturbation of the associated subspace. Given the Lipschitz continuity of the underlying field with respect to the spatial domain, it is reasonable to assume that neighbouring nodes are determined by ridge directions that are closely related to each other. This motivates the proposal of algorithms to compress the representation of an embedded ridge function by approximating the ridge directions of some nodes as a function of their neighbours. After compression, only a fraction of the original ridge directions need to be stored.

In passing, we note that the actual approximation error incurred via compression can be smaller than suggested by \Cref{thm:subspace_perturb}, since the change in the ridge profile $g$ as a result of the perturbation in the subspace is not accounted for. In practice, after approximating the subspace by a perturbed version of its original value, the ridge profile can be refitted to data projected to the new subspace, minimising the MSE in the process. The new error can be smaller than simply applying $g$ to the data projected to the new subspace without changing the compression level.

\subsection{Ridge compression and recovery} \label{sec:sparse_algs}
Given \emph{a priori} knowledge of the relationship between neighbouring ridge directions, we can avoid storing the ridge directions for all output components. This is useful for re-creating the PDE-scalar field from the selected nodes. To this goal, we propose the \emph{ridge compression and recovery} algorithms. The former allows us to retain only a subset of suitably subsampled output nodes (components of the vector $\vf$); the latter recovers the remaining nodes from these subsamples.

Our algorithm for ridge compression is detailed in \Cref{algo_ave}, where each removed component will be reconstructed by the average of two of its closest neighbours. Given an embedded ridge approximation and the number of components that need to be removed $k$, we iterate through all the nodes $N$ and identify two neighbours for each node. These neighbours are identified based on the smallest subspace distance (see \eqref{equ:subspace_distance}) between successive nodes; see steps 6 and 7. In step 7, we require that the second closest neighbour must be closer to the candidate to be removed than the first neighbour in step 6. If this step is not enforced, the average between the neighbours is a poor approximation of the removed candidate. Following this, in step 10, we sort the candidates for removal by considering the sum of the distances of the removal candidates to their two neighbours. Removing a candidate with smaller total distance is prioritised over removing one with larger total distance. From step 11 onwards, we attempt to remove the candidates, according to the order determined in step 10. The recovery algorithm (\Cref{algo_ave_recov}) reconstructs the missing components based on the list of nearest neighbours (the output from \Cref{algo_ave}). We only consider the case where $r=1$ here, permitting us to easily estimate the missing node's ridge subspace as a linear combination of the neighbouring components.
{\SetAlgoNoLine%
\begin{algorithm}[h]
\caption{Ridge compression algorithm for embedded ridge approximations.}
\label{algo_ave}
\begin{algorithm2e}[H]
\Input{List of ridge directions $\mW_{1},...,\mW_{N}$ corresponding to  $g_1,...,g_N$ (but the ridge profiles are not needed), and the number of components to retain $k$.}
\Output{List of subsampled ridge directions $\mW_{N_1},...,\mW_{N_k}$, and a list of nearest neighbours $L \in \mathbb{N}^{(N-k)\times 2}$ corresponding to missing components $m \in \mathbb{N}^{N-k}$. }
Initialise empty list $m$ and array $L$, and $I_s = (1,...,N)$.\\
\While {length of $m$ is smaller than $N-k$ and $I_s \neq \varnothing$}
{
	$I' = I_s \backslash (m\cup L)$\footnote{Note that we convert the lists to sets before we perform set operations on them; that is, we remove duplicate elements and no longer enforce the order which was present in the list, and for a two-dimensional array we flatten the array and consider all distinct elements.} \Comment{Gather the remaining non-removed, non-paired components.}\\
	$A = L \cup I'$. \Comment{Gather the available neighbours.}\\
	\For {$i$ up to the length of $I'$}
	{
	$L'[i,1] = \text{argmin}_{j\in A \backslash i} dist(\mW_i, \mW_j)$ \Comment{Find the best neighbours for each index.}\\
	$L'[i,2] = \, \text{argmin}_{j\in A \textbackslash \{L'[i,1], i\}} dist(\mW_i, \mW_j)$\\
	\nonl \quad subject to $dist(\mW_i, \mW_j) <  dist(\mW_j, \mW_{L'[i,1]})$\\
	$D'[i]= dist(\mW_i, \mW_{L'[i,1]}) + dist(\mW_i, \mW_{L'[i,2]}).$ \Comment{Compute total distances.}\\
	}
	Sort $I'$ and $L'$ columnwise in ascending order of $D'$ to give $I_s$ and $L_s$.\\
	\For {$i$ up to the length of $I_s$}
	{
		\If {$I_s[i] \notin m \cup L$ and $L_s[i,1], L_s[i,2] \notin m$}
		{
			$m \leftarrow m \cup \lbrace I_s[i] \rbrace$.\\
			$L \leftarrow L \cup \lbrace L_s[i,1], L_s[i,2] \rbrace$.\\
		}
	}
}
$k = N - \text{length}(m)$.\footnote{This may be larger than the input $k$ because no further components can be removed without compromising the neighbours of the already removed ones.}\\
$(N_1,\ldots,N_k) = (1,...,N)\backslash m$.\\
Retain $\mW_{N_1},...,\mW_{N_k}$ and discard the rest.\\
\end{algorithm2e}
\end{algorithm}}

{\SetAlgoNoLine%
\begin{algorithm}[h]
\caption{Ridge recovery algorithm for a compressed embedded ridge approximation.}
\label{algo_ave_recov}
\begin{algorithm2e}[H]
\Input{Subsampled points $(N_1,...,N_k)$, list of matrices $\mW_{N_1}, \ldots ,\mW_{N_k} \in \myRe^d$, and a list of nearest neighbours $L\in \mathbb{N}^{(N-k)\times 2}$ corresponding to missing components $m \in \mathbb{N}^{N-k}$.}
\Output{$\mU_1, \ldots ,\mU_N$ }
\For{$i=1,...,N$}
{
	\eIf{$i\in (N_1,\ldots,N_k)$}
	{
		$\mU_i = \mW_i$.
	}
	{	
		Find $j$ such that $m[j] = i$.\\
		$\mU_{i1}' = \mW_{L[j,1]} + \mW_{L[j,2]}$\\ 
		$\mU_{i2}' = \mW_{L[j,1]} - \mW_{L[j,2]}$\\
		$P = \text{argmin}_{p=1,2} dist(\mU_{ip}',\mW_{L[j,1]})$\\
		$\mU_i = \text{column normalize}(\mU_{iP}')$
	}
}
\end{algorithm2e}
\end{algorithm}}

In the ridge compression algorithm, once a node is marked as one of the neighbours of a removed node, it can no longer be removed. This sets a hard limit on how many nodes can be removed before all stored nodes are marked. One way to circumvent this difficulty is to apply the compression and recovery algorithms \emph{recursively}. \Cref{fig:compression_recursive} illustrates this idea. At each compression stage, up to $S$ components are compressed, where $S$ can be set by the user. After this, the remaining components are fed to the next stage as input to remove up to a further $S$ components and so on. To recover the removed components, the recovery algorithm is applied stagewise, similarly to the compression process, in the reverse direction. Note that upon passing the remaining components to the next stage, even though some of the remaining components are neighbours to removed components in the previous stage, it is possible to remove them in the next stage. This is because their ridge directions are not required until the corresponding stage in the recovery process, when these components will have been reconstructed by the previous recovery stage. In this way, the degree of compression can be increased.

\begin{figure}
\includegraphics[width=\linewidth]{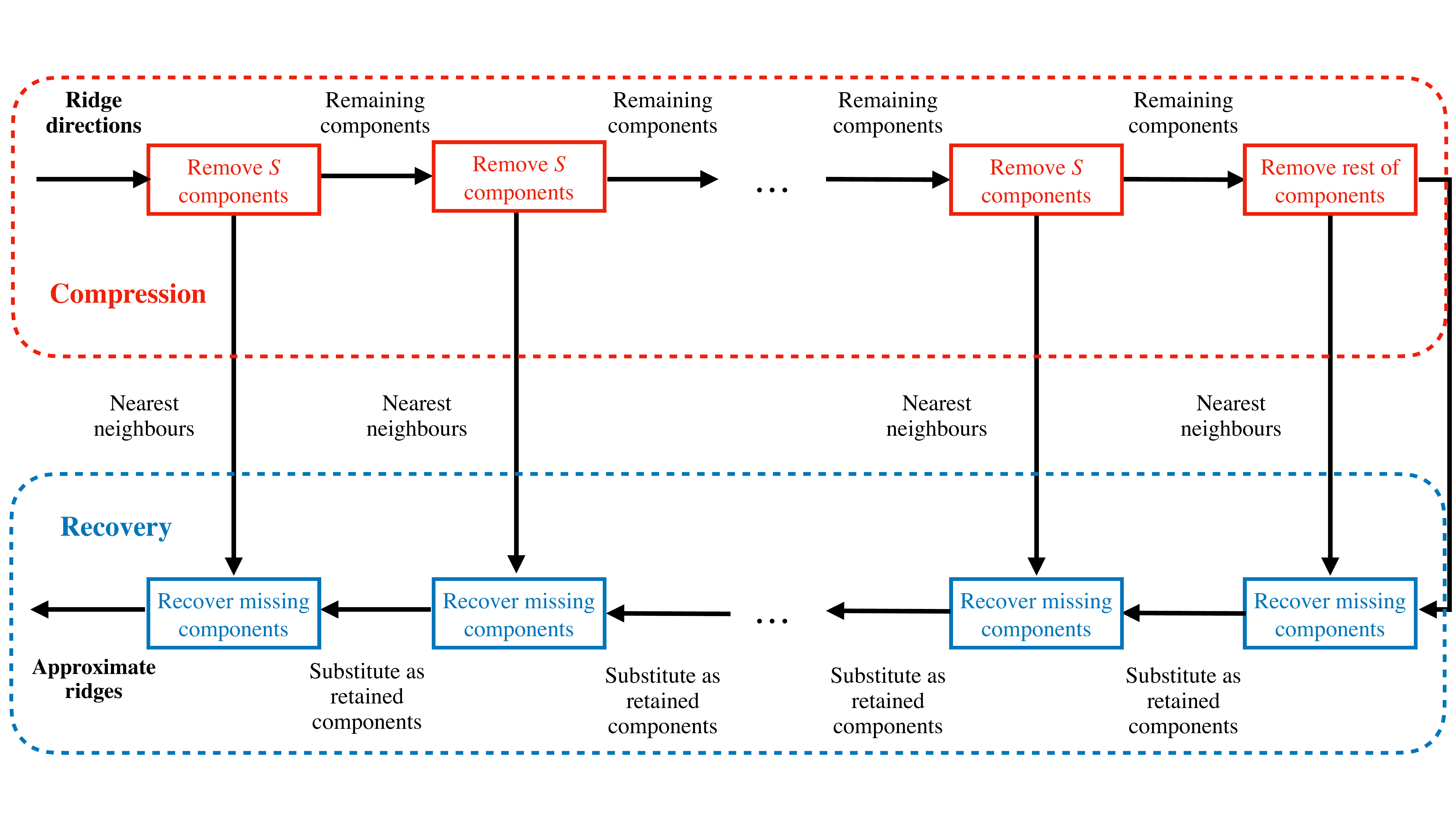}
\caption{Schematic for applying the ridge compression algorithm (\Cref{algo_ave}) and recovery algorithm (\Cref{algo_ave_recov}) recursively, removing at most $S$ components at a time.}
\label{fig:compression_recursive}
\end{figure}

Both the ridge compression and recovery algorithms presented in this section are greedy algorithms, which may therefore not result in the storage configuration that globally minimises the distance between the missing components and their neighbours. However, to determine the globally optimal solution requires a combinatorial search over every storage configuration, which is computationally prohibitive. 

Note that the compression problem can be interpreted as a clustering task, where cluster centres are retained and all other ridges can be recovered by identifying each with the closest cluster or a linear combination of the two closest centres. Operating with a non-Euclidean metric defined by the subspace distance, clustering algorithms such as $k$-medoids can be used. \Cref{alg:k-medoids} describes an algorithm for clustering ridge directions using $k$-medoids, based on its implementation in \cite{park2009simple}. The corresponding recovery algorithm can be similar to \Cref{algo_ave_recov}. Alternatively, each removed component can be replaced by its nearest medoid. In \Cref{sec:storage_results}, our  algorithm is compared with $k$-medoids for compressing the flow field of a CFD simulation case study.

{\SetAlgoNoLine%
\begin{algorithm}[h]
\caption{Clustering ridge directions with $k$-medoids.}
\label{alg:k-medoids}
\begin{algorithm2e}[H]
\Input{List of ridge directions $\mW_{1},...,\mW_{N}$ corresponding to  $g_1,...,g_N$ (but the ridge profiles are not needed), and the number of components to retain $k$.}
\Output{List of subsampled ridge directions $\mW_{N_1},...,\mW_{N_k}$, and a list of nearest neighbours $L \in \mathbb{N}^{(N-k)\times 2}$ corresponding to missing components.}
Initialise list of medoids $D \subset (1,...,N)$ randomly, where the number of medoids is equal to $k$.\\
Assign each non-medoid component to its closest medoid.\\
$\Sigma_d = $ sum of distances of each non-medoid component to its nearest medoid.\\
\While{the value of $\Sigma_d$ is different from its value in previous iteration}
{
Find a new medoid from each cluster which minimises the sum of distances to all other components in the cluster.\\
Assign each non-medoid component to its new closest medoid and calculate $\Sigma_d$.\\
}
\For{each non-medoid component $\mW_i$}
{
$L[i,1] = \text{argmin}_{j\in D \backslash i} dist(\mW_i, \mW_j)$\\
$L[i,2] = \, \text{argmin}_{j\in D \textbackslash \{L'[i,1], i\}} dist(\mW_i, \mW_j)$\\
	\nonl \quad subject to $dist(\mW_i, \mW_j) <  dist(\mW_j, \mW_{L[i,1]})$\\
}
Retain $\mW_{N_1},...,\mW_{N_k}$ where $D = (N_1,...,N_k)$ are the medoids, and discard the rest.\\
\end{algorithm2e}
\end{algorithm}}

%% file: sec-4.tex
\section{Numerical examples} \label{sec:examples}
In this section, we illustrate the embedded ridge approximation approach with an analytical example and a CFD example.
\subsection{Analytical example} \label{sec:analytical_ex}
Consider the function
\begin{equation} \label{equ:analytical_ex}
h(\vx) = [2\quad 3\quad 5]\begin{bmatrix}
f_1(\vx)\\
f_2(\vx)\\
f_3(\vx)
\end{bmatrix}
\end{equation}
where 
\begin{align*}
f_1(\vx) &= \left(\vw_1^T \vx\right)^2 + \left(\vw_1^T \vx\right)^3,\\
f_2(\vx) &= \exp\left(\vw_2^T \vx\right),\\
f_3(\vx) &= \sin\left(\left(\vw_3^T \vx\right) \pi \right),
\end{align*}
defined over the domain $\mathcal{D} = [-1,1]^{10}$ where the inputs $\vx$ are independent and have uniform marginals. Note that $h(\vx)$ is an exact ridge function with three ridge directions spanned by the columns of $\mU = [\vw_1, \vw_2, \vw_3]$. We draw $\vw_1, \vw_2, \vw_3$ as random vectors with unit Euclidean norm and compare the recovered ridge directions to the drawn vectors using the subspace distance (see \eqref{equ:subspace_distance}). In this example, polynomial variable projection (VP) is used for finding ridge directions, where the polynomials have a maximum total degree of 7. For the optimisation loop inside the algorithm for VP (see \cite[Algorithm 4.1]{hokanson2018data-driven}), we set the convergence criterion to be when the subspace distance between the ridge directions of the previous and current iterations is smaller than $10^{-7}$. Our implementation of this algorithm can be found in the \emph{Effective Quadratures} open-source library \cite{seshadri2017effective-quadratures} (\url{https://www.effective-quadratures.org/}). 

For embedded ridge approximation, we use VP to estimate the ridge directions for each component function $f_1(\vx), f_2(\vx)$ and $f_3(\vx)$, and then calculate the first three leading eigenvectors of the vector gradient covariance matrix of $\vf(\vx) = [f_1(\vx), f_2(\vx), f_3(\vx)]^T$. The weights are set as $\boldsymbol{\omega} = [2, 3, 5]^T$ to find an estimate of the dimension-reducing subspace of $h(\vx)$. For direct ridge approximation, we use VP to estimate the three-dimensional dimension-reducing subspace of $h(\vx)$ directly. We vary the number of observations used for each method and examine the subspace distance between the recovered directions and the true directions. We note that the results are binary---we either get a small subspace distance from successful recovery or a large subspace error from failure in recovery. Thus, we plot the probability of successful recovery---where the subspace distance is below 0.005---across 40 trials on the left of \Cref{fig:analytical_ex}. This plot shows that recovery using embedded ridge approximations is more stable and requires fewer observations than direct ridge approximation for a given recovery probability. 

To achieve successful recovery of $\mU$ from the embedded ridge approximation, we need to be able to successfully recover the ridge directions in each individual function, as reflected from the right plot of \Cref{fig:analytical_ex}. Interestingly, despite the need to successfully find three sets of ridge directions concurrently, the probability of recovery is still significantly higher for the embedded ridge approximation method. This is because the optimisation over three-dimensional subspaces required in the direct method is much more challenging than their one-dimensional counterparts required in the embedded method (see Table 3 in \cite{hokanson2018data-driven}).

\begin{figure}
\begin{center}
\begin{minipage}[b]{0.475\linewidth}
\includegraphics[width=1\linewidth]{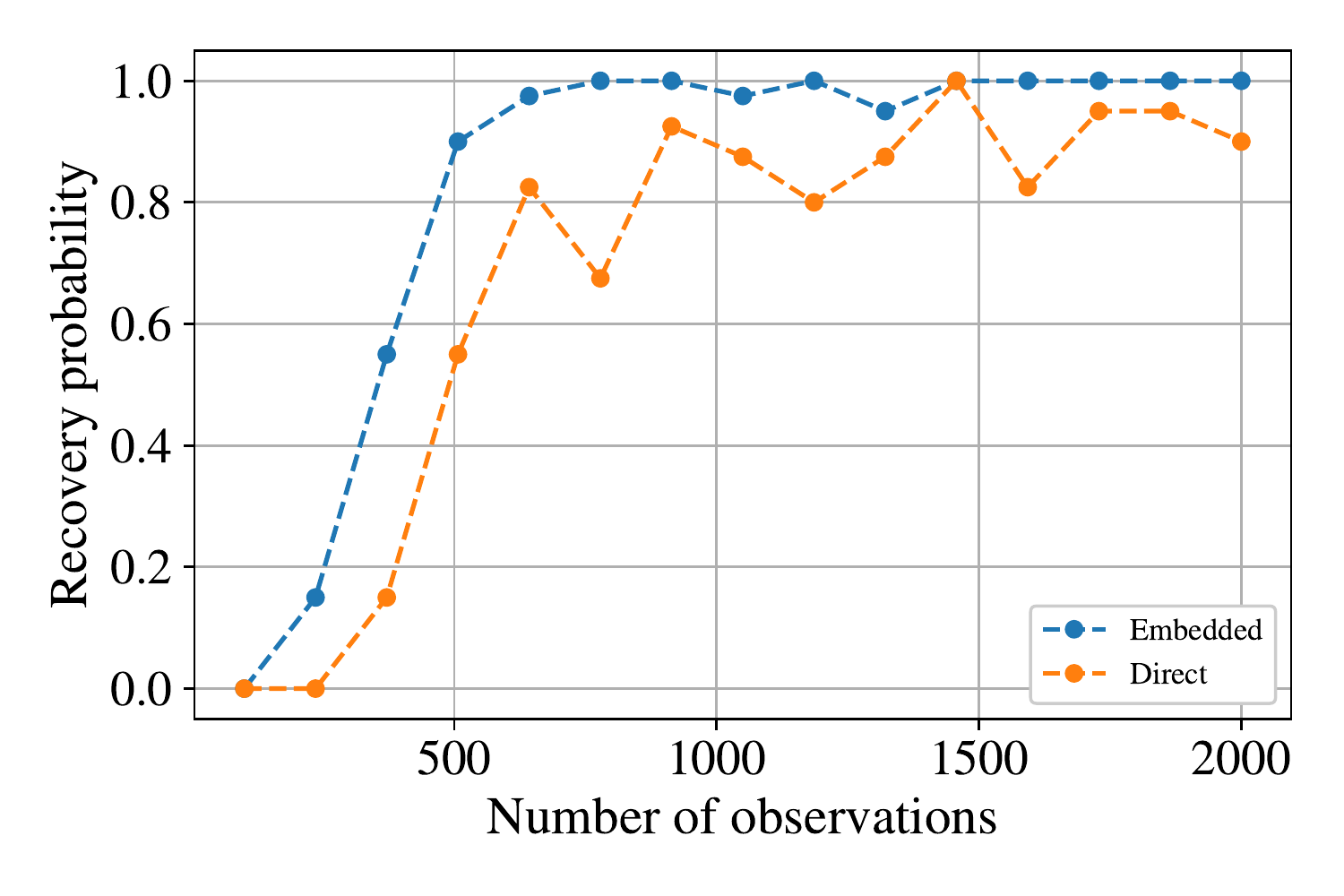}
\end{minipage}
\begin{minipage}[b]{0.475\linewidth}
\includegraphics[width=1\linewidth]{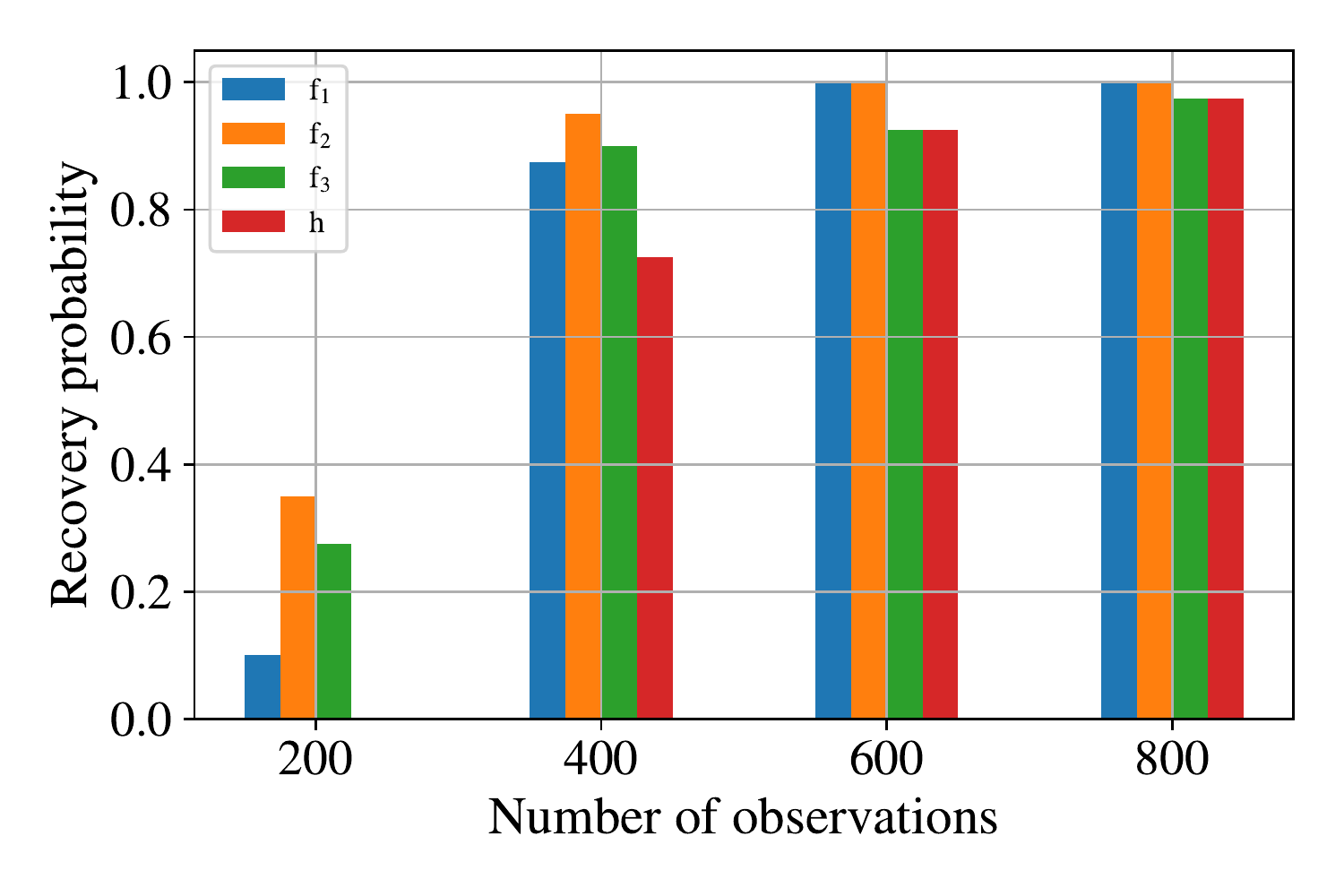}
\end{minipage}
\caption{(Left) Comparing the average recovery probability for embedded and direct ridge approximation for $h(\vx)$ in \eqref{equ:analytical_ex}. (Right) Recovery probability of component ridges and qoi ridge when using an embedded ridge approximation. A successful recovery is defined to be when the subspace error is smaller than 0.005. Forty trials are performed. }\label{fig:analytical_ex}
\end{center}
\end{figure}

\subsection{Shape design of the NACA0012}
We apply the embedded ridge function approximation algorithm (see \Cref{alg:embedded_ridge_comp}) to the shape design of the NACA0012 airfoil. The shape deformation of the baseline NACA0012 profile is parameterised using $d=50$ Hicks-Henne bump functions around the airfoil, and the variation in the surface pressure profile is measured. We fix an entry Mach number of 0.3 (subsonic) and an angle of attack of 1.25\textdegree, with free-stream temperature and pressure at 273.15 K and 101325 Pa, respectively. The pressure profile is solved using the compressible Euler flow solver in the open source CFD suite $SU^2$ \cite{economon2016su2:}. The coefficients of lift and drag\footnote{Ignoring skin friction and assuming a unit reference area.} are known to be linear functions of the pressure around the airfoil \cite[Ch.~1]{anderson2010fundamentals}, given by

\begin{minipage}[b]{0.45\linewidth}
\begin{align*}
C_l &= \frac{1}{\frac{1}{2}\rho_0 v_{\infty}^2}\oint p(\mb{x}) \mb{n}\cdot \mb{k}\; \text{d}s\\
&\approx \frac{1}{\frac{1}{2}\rho_0 v_{\infty}^2} \sum_{i=1}^N p_i(\mb{x})  \mb{n}_i \cdot \mb{k} \;\Delta s_i\\
&= \boldsymbol{\omega}_l^T \mb{p}(\mb{x}),
\end{align*}
\end{minipage}
\begin{minipage}[b]{0.45\linewidth}
\begin{align*}
C_d &= \frac{1}{\frac{1}{2}\rho_0 v_{\infty}^2}\oint p(\mb{x}) \mb{n}\cdot \mb{j} \; \text{d}s\\
&\approx \frac{1}{\frac{1}{2}\rho_0 v_{\infty}^2} \sum_{i=1}^N p_i(\mb{x})  \mb{n}_i \cdot \mb{j} \;\Delta s_i\\
&= \boldsymbol{\omega}_d^T \mb{p}(\mb{x}),
\end{align*}
\end{minipage}

\noindent
where the integral is evaluated around the airfoil surface, spatially parameterised by $s$. The input variable $\vx \in \myRe^d$ contains the Hicks-Henne bump amplitudes; $\mb{n}$ is the surface normal, $\mb{k}$ the direction perpendicular to the flow, and $\mb{j}$ the direction parallel to the flow (see \cref{fig:airfoil_local_models}). In the normalising factors, $\rho_0$ is the free-stream density, and $v_{\infty}$ is the free-stream speed. We discretise this problem by considering $N=200$ measurements of pressure around the airfoil, resulting in the vector-valued function $\mb{p}: \myRe^d \rightarrow \myRe^N$ representing the surface pressure profile. Note that the approximation in the second line of both expressions comes not only from the discretisation but also from the assumption that $\vn$ is independent of $\vx$---a good approximation when the geometric perturbations are small. Under this approximation, the coefficients of lift and drag can then be expressed as linear functions of the components of $\mb{p(x)}$.

\begin{figure}
\includegraphics[width=\linewidth]{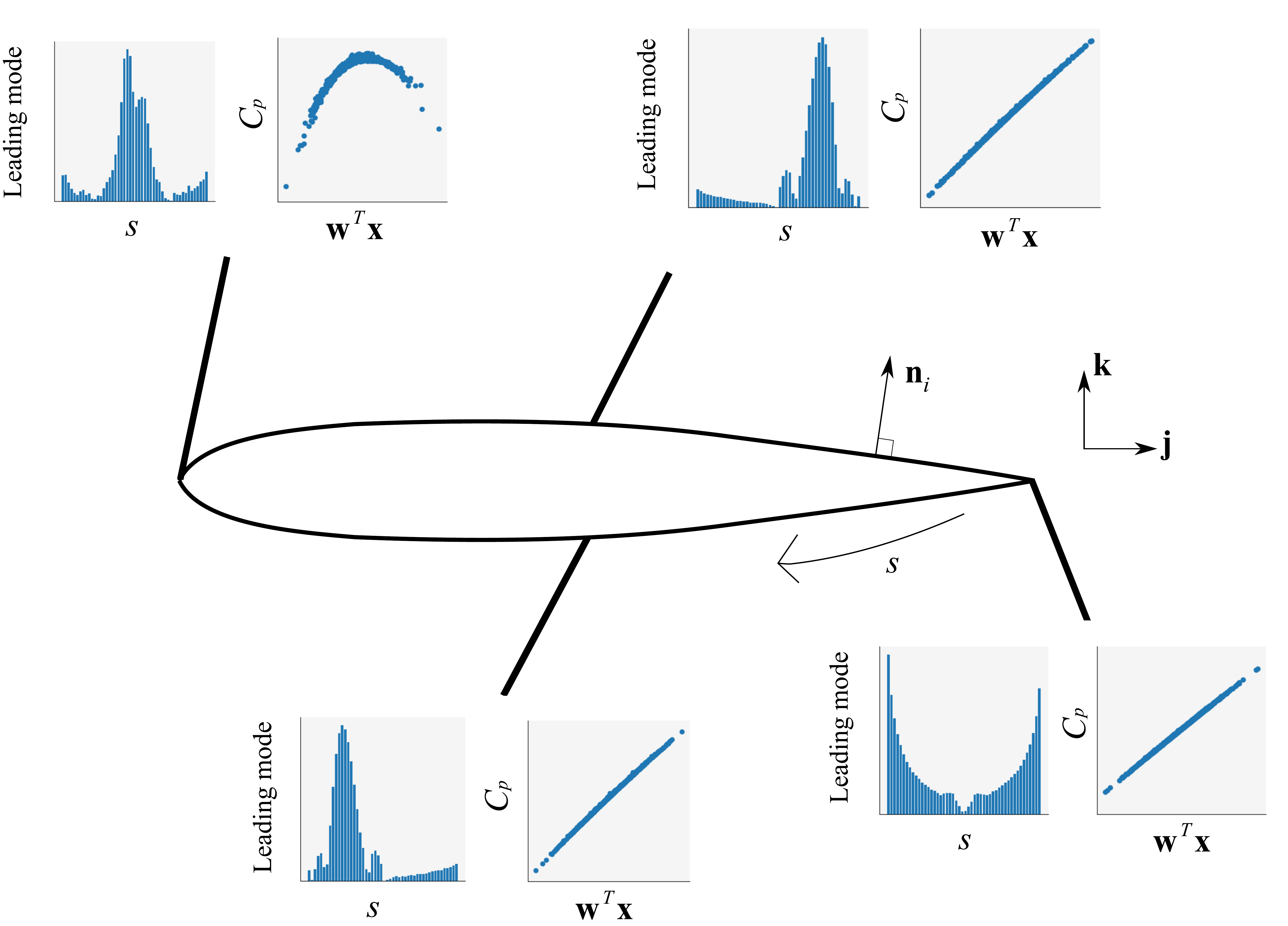}
\caption{Fitting a one-dimensional polynomial ridge function for each pressure component. The left-hand plot of each pair is the magnitude of the leading mode against the airfoil coordinate ($s$); the right-hand plot is a sufficient summary plot \cite{cook1998regression} at each location.}
\label{fig:airfoil_local_models}
\end{figure}

As the flow is entirely subsonic and inviscid, we expect the bumps to have a strongly local influence. Hence, the pressure profile $\mb{p(x)}$ is well-approximated by an embedded ridge function. This motivates the following approach to estimate $C_l$ and $C_d$, which applies the steps in \Cref{alg:embedded_ridge_comp} assuming each node is approximated by a one-dimensional ridge function.

\begin{enumerate}
\item Using a gradient-free computational strategy, estimate the leading ridge direction $\widehat{\vw}_i$ for each $p_i(\vx)$. 
\item Fit a low-dimensional surrogate using this leading mode for each $p_i$. That is, we seek
\begin{equation}
p_i(\mb{x}) \approx \widehat{g}_i(\widehat{\vw}_i^T\mb{x}),
\end{equation}
for the $i$-th component of $\mb{p}$. We use univariate orthogonal polynomials for the profiles $\widehat{g}_i(\cdot)$.
\item We can compute the elements of the Jacobian via these ridge approximations:
\begin{equation}
\widehat{\mJ}(\mb{x})_{ij} = \widehat{w}_{ji} \widehat{g}_j'(\widehat{\mb{w}}_j^T\mb{x}),
\end{equation}
where $\widehat{w}_{ji}$ is the $i$-th element of $\widehat{\mb{w}}_j$. Gradients here are furnished by the polynomial approximation analytically.
\item Compute the gradient covariance matrix with \eqref{equ:MC_Ch} by substituting $\boldsymbol{\omega}_l$ and $\boldsymbol{\omega}_d$, from which we can compute the dimension-reducing subspaces and form ridge approximations of the scalar qois (the coefficients of lift and drag respectively).
\end{enumerate}

\subsubsection{Results}

To apply the embedded ridge approximation approach, we will fit a one-dimensional ridge function for each component of the surface pressure profile, $p_i$, and use a quadratic ridge profile for each component. Three gradient-free dimension-reducing strategies to find the ridge subspaces at each component are studied:

\begin{enumerate}
\item Fitting global linear models for each node, and taking the ridge direction as the normalised parameters of the linear model \cite[Algorithm 1.3]{constantine2015active}---see Appendix \ref{sec:lin_models} for further details. Note that the \emph{ridge profiles} are still quadratic; the linear models are only used to find the ridge directions. This will be referred to as ``Embedded linear''.
\item As above, but using quadratic polynomial VP only for nodes close to the leading edge, noting that pressure variation near the leading edge tends to be non-linear. The ridge subspace remains one-dimensional for all nodes. This will be referred to as ``Embedded VP''.
\item As above, but using MAVE \cite{xia2002adaptive} only for nodes close to the leading edge to extract a one-dimensional ridge subspace. Then, a quadratic polynomial is fitted in this one-dimensional subspace for these nodes. This will be referred to as ``Embedded MAVE''.
\end{enumerate}
The implementation for VP is the same as in \Cref{sec:analytical_ex}. For MAVE, we adapt the R code from Hang and Xia \cite{hang2018mave:}. A brief exposition on the MAVE method is provided in Appendix \ref{sec:MAVE}. Once the nodal ridge approximations are formulated, the ridge approximations for the qois $C_l$ and $C_d$ are furnished via \Cref{alg:embedded_ridge_comp}.

The embedded ridge approximation approach is compared with the direct ridge approximation approach, where observations for $C_l$ and $C_d$ are used to find a ridge approximation directly and without the use of gradients. For the direct approach, three dimension-reducing strategies are studied---VP, MAVE and the linear model. For both the embedded and direct approaches, one dimension is used for the ridge approximation of $C_l$, and two dimensions for $C_d$. Note that the linear model used in the direct approach is unable to estimate more than one dimension, so only one is used for $C_d$ in this case.\footnote{It is possible to construct surrogate models for $C_l$ and $C_d$ based on nodal ridge approximations alone, as noted in \Cref{sec:Embedded_ridge_functions}. However, as noted in the same section, the ridge subspace has utility on its own. Moreover, in this case study, evaluating the ridge approximation for the qois sets up a better comparison with direct ridge approximations.}

In \Cref{fig:emu_err}, we plot the MSE of the surrogate model fitted using the dimension-reducing subspaces resulting from embedded and direct ridge approximation. The MSE of approximating the qoi $h(\vx)$ with $\widehat{h}(\vx)$ is evaluated as
\begin{equation} \label{equ:MSE_metric_qoi}
\epsilon_{h} = \frac{1}{M'} \sum_{j=1}^{M'} \frac{ \left(   h(\vy^{(j)}) - \widehat{h}(\vy^{(j)})  \right) ^2}{\sigma_{h}^2},
\end{equation}
where $\vy^{(j)}\in \myRe^d$ are verification samples drawn independently from data used to train the response surfaces. The ridge approximation of $h(\vx)$ evaluated at $\vy^{(j)}$ is denoted $\widehat{h}(\vy^{(j)}) = \widehat{g}_i\left(\widehat{\mU}^T \vy^{(j)}\right)$, and $\sigma_{h}^2$ is the sample variance of $h(\vy)$ across all $M'$ verification samples. 

It is shown that using embedded ridge approximation reduces the MSE compared to direct estimation when the number of samples is limited. The errors for embedded VP and MAVE approximately reach convergence in 300 observations for $C_l$, and 400 observations for $C_d$. Although the linear models (in both the direct and embedded cases) suffice for estimating $C_l$, for functions with stronger non-linear dependencies such as $C_d$, the linear model is shown to have a larger error compared to VP and MAVE. We also note that the use of embedded ridge approximation permits us to extend the capability of linear models to estimate more than one mode in the scalar qois, improving its performance as seen on the right of \Cref{fig:emu_err}.

\begin{figure}
\begin{center}
\begin{minipage}[b]{0.475\linewidth}
\includegraphics[width=1\linewidth]{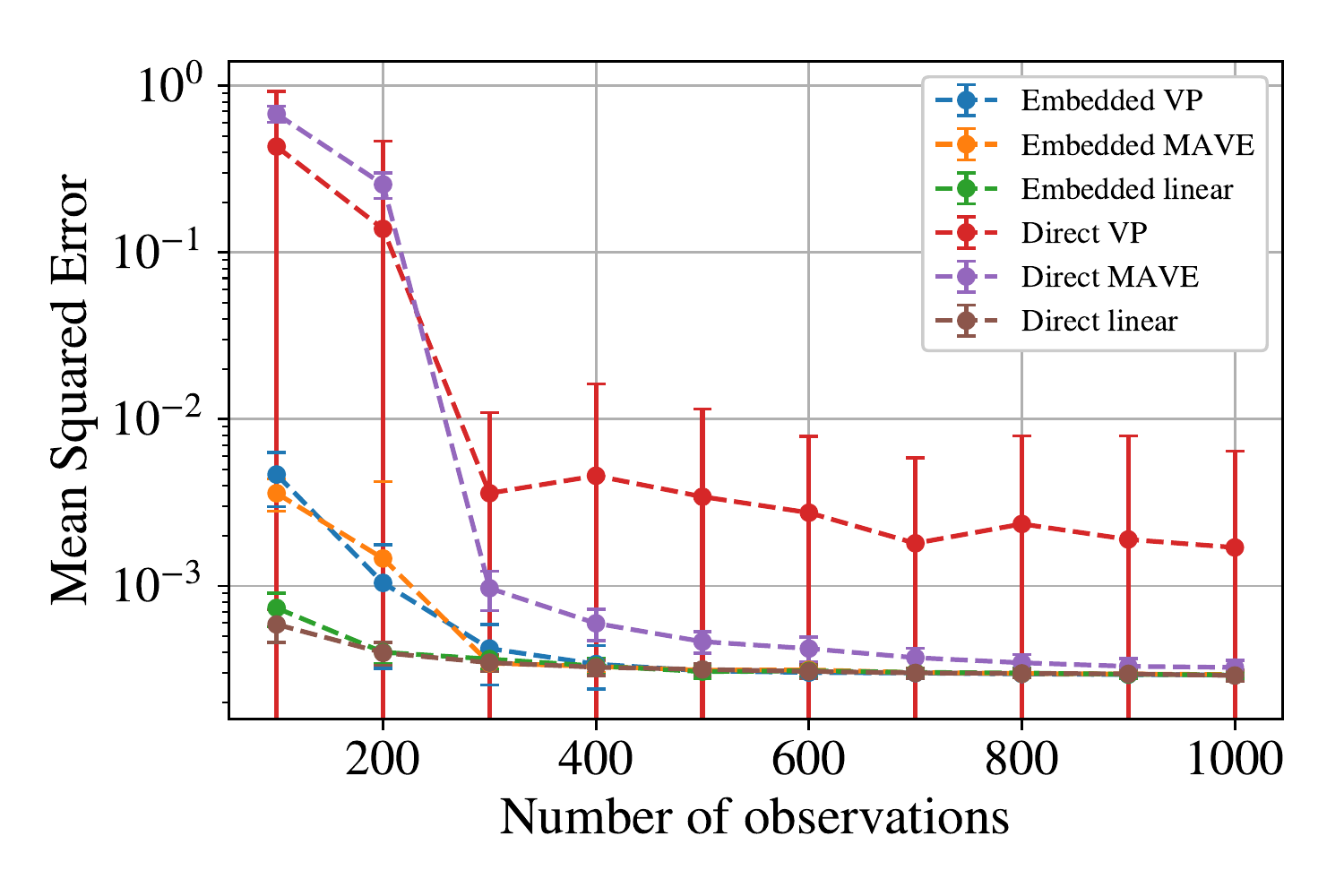}
\end{minipage}
\begin{minipage}[b]{0.475\linewidth}
\includegraphics[width=1\linewidth]{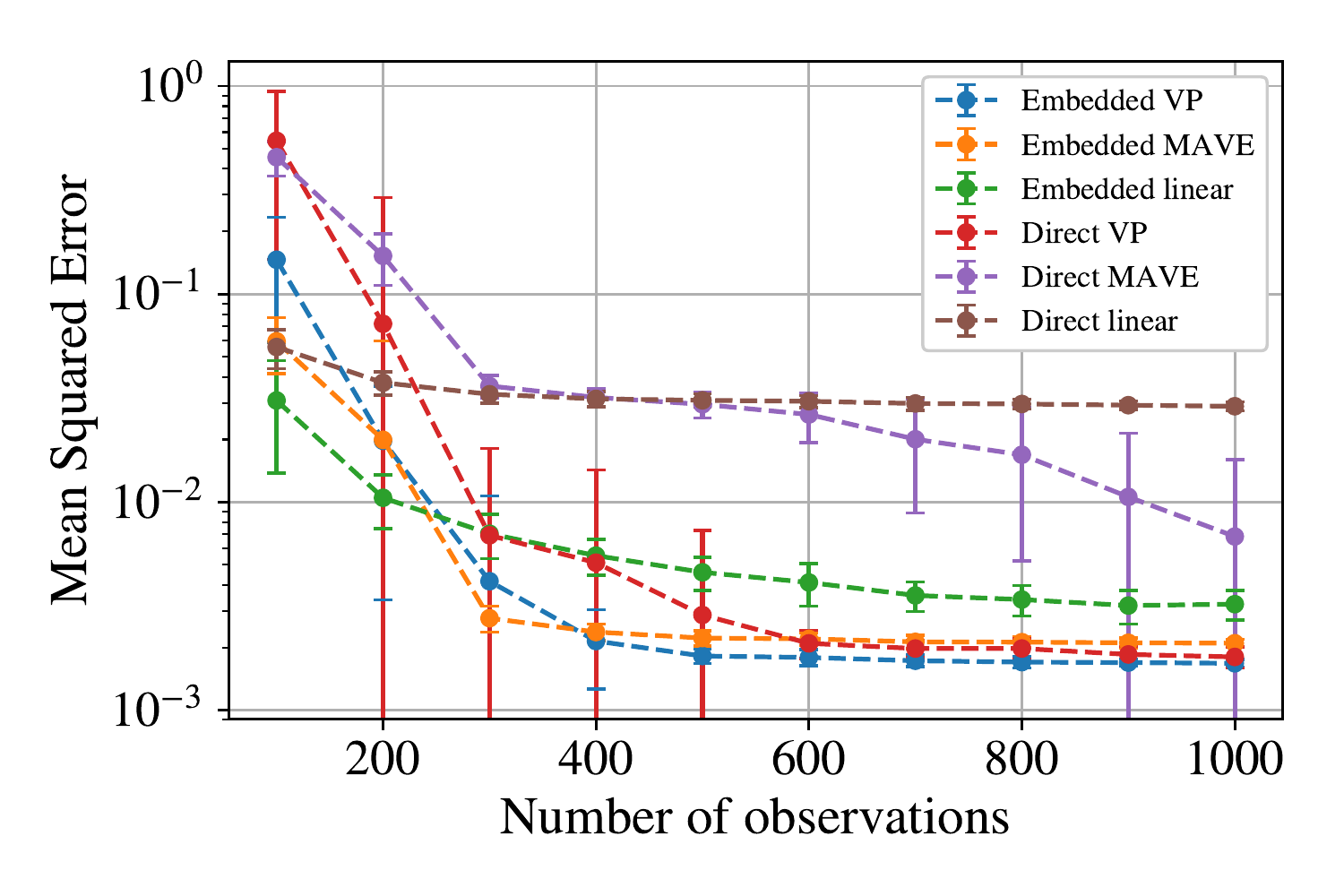}
\end{minipage}
\caption{Mean squared error of $C_l$ (left) and $C_d$ (right) for surrogate models with VP, MAVE and linear models via direct and embedded ridge approximations.}\label{fig:emu_err}
\end{center}
\end{figure}

\subsection{Sparse storage of NACA0012 pressure field}
\label{sec:storage_results}

In this subsection, we demonstrate the application of the ridge compression and recovery algorithms described in \Cref{sec:storage} on the estimation of the pressure field around the NACA0012 airfoil. The flow conditions and perturbation variables are set exactly the same as in the previous subsection. In the flow solution, the computational domain is discretised into $N=5233$ nodes. It is observed that each node in the flow field can be well-approximated by a one-dimensional ridge function. We examine the efficacy of our storage and recovery algorithms by selecting a subset of the components to store and attempting to recover the missing components from the stored ones.

We remove a range of numbers of components using the ridge compression algorithm (\Cref{algo_ave}) applied recursively, where in each stage at most $S=500$ components are removed. Then, we reconstruct the missing ridge directions using the ridge recovery algorithm (\Cref{algo_ave_recov}), again applied recursively. The reconstruction quality is evaluated using the average normalised MSE $\epsilon_R$, defined as
\begin{equation} \label{equ:MSE_metric}
\epsilon_R = \frac{1}{N'M'} \sum_{i=1}^{N'} \sum_{j=1}^{M'} \frac{\left(p_i\left(\vy^{(j)}\right) - \widehat{p}_i\left(\vy^{(j)}\right)\right)^2}{\sigma_{pi}^2},
\end{equation}
where $N'$ is the number of recovered components, and other variables are defined similarly as before. For comparison, the $k$-medoids clustering algorithm (\Cref{alg:k-medoids}) is run under the same settings, with the same recovery algorithm (\Cref{algo_ave_recov}). In addition, a random deletion strategy is run, where the removed nodes are selected randomly, and missing modes are recovered by substituting the nearest neighbour in terms of subspace distance.

In \Cref{fig:storage_results}, the MSE averaged across all recovered components is plotted for the three methods: the ridge  compression algorithm, $k$-medoids clustering and random deletion. The plot shows that applying compression recursively allows recovery of missing ridge subspaces with greater accuracy than the other methods up to approximately 4700 components, which covers almost all of the nodes. 

In \Cref{fig:Cp_profiles} and \Cref{fig:contour}, the $C_p$ profile on the surface of the airfoil and the entire flow field are compared for two cases: full CFD results and the reconstruction after removal of 3000 nodes using the compression algorithm respectively. The plots are for an airfoil geometry which was not used in the computation of the embedded ridge approximation. It can be seen that the pressure field is approximated well. Near the leading edge where large pressure variations are present, the pressure is also well-estimated. \Cref{fig:mesh} shows the locations of the removed nodes at different levels of compression. Nodes at the far-field are prioritised for removal, and, at high compression levels, nodes near the leading and trailing edges tend to be retained.

\begin{figure}
\centering
\includegraphics[width=0.6\linewidth]{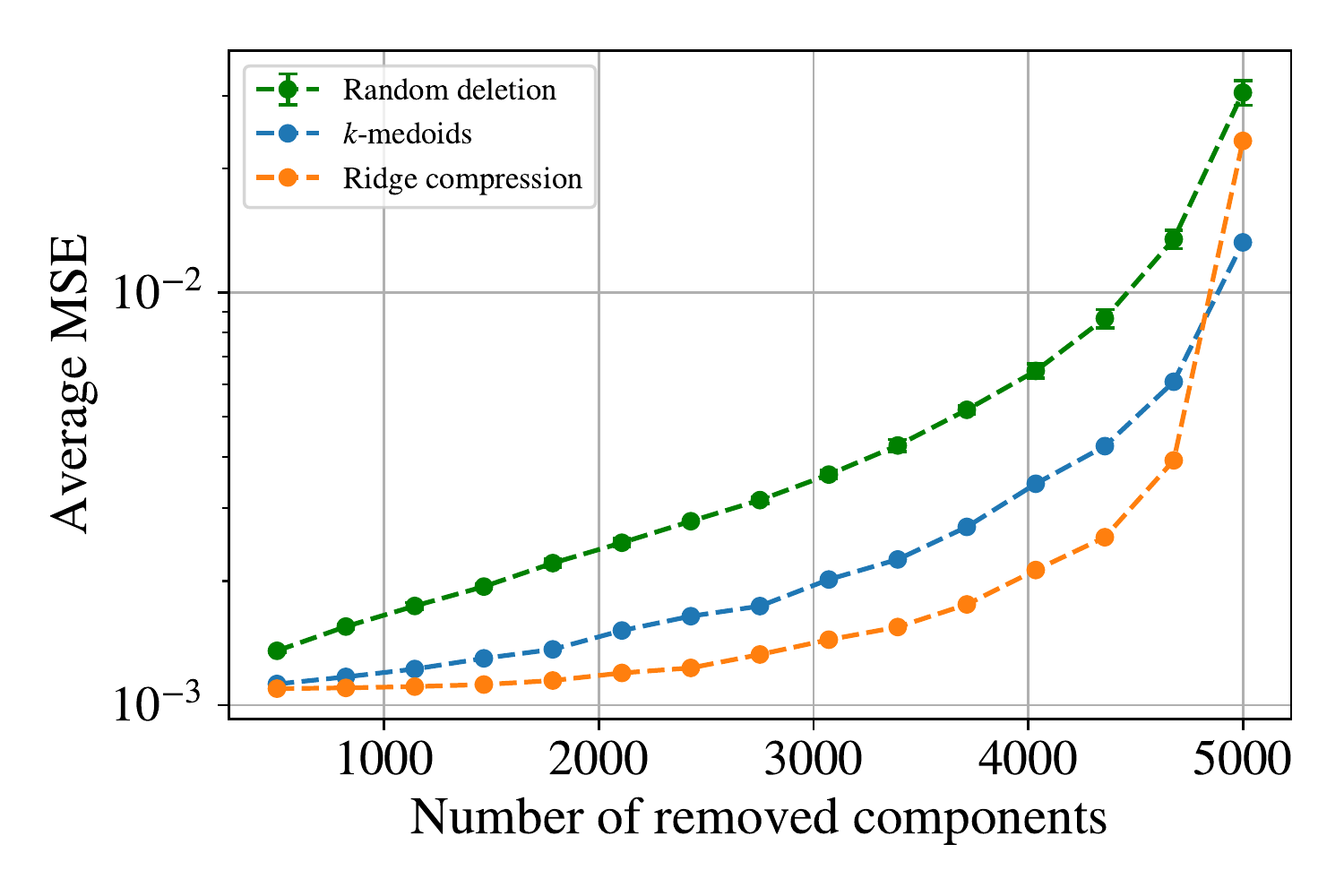}
\caption{Average MSE after removing various numbers of field components using the ridge compression algorithm (\Cref{algo_ave}) applied recursively with a stride $S=500$, $k$-medoids clustering (\Cref{alg:k-medoids}) and random deletion.}
\label{fig:storage_results}
\end{figure}

\begin{figure}
\centering
\includegraphics[width=0.6\linewidth]{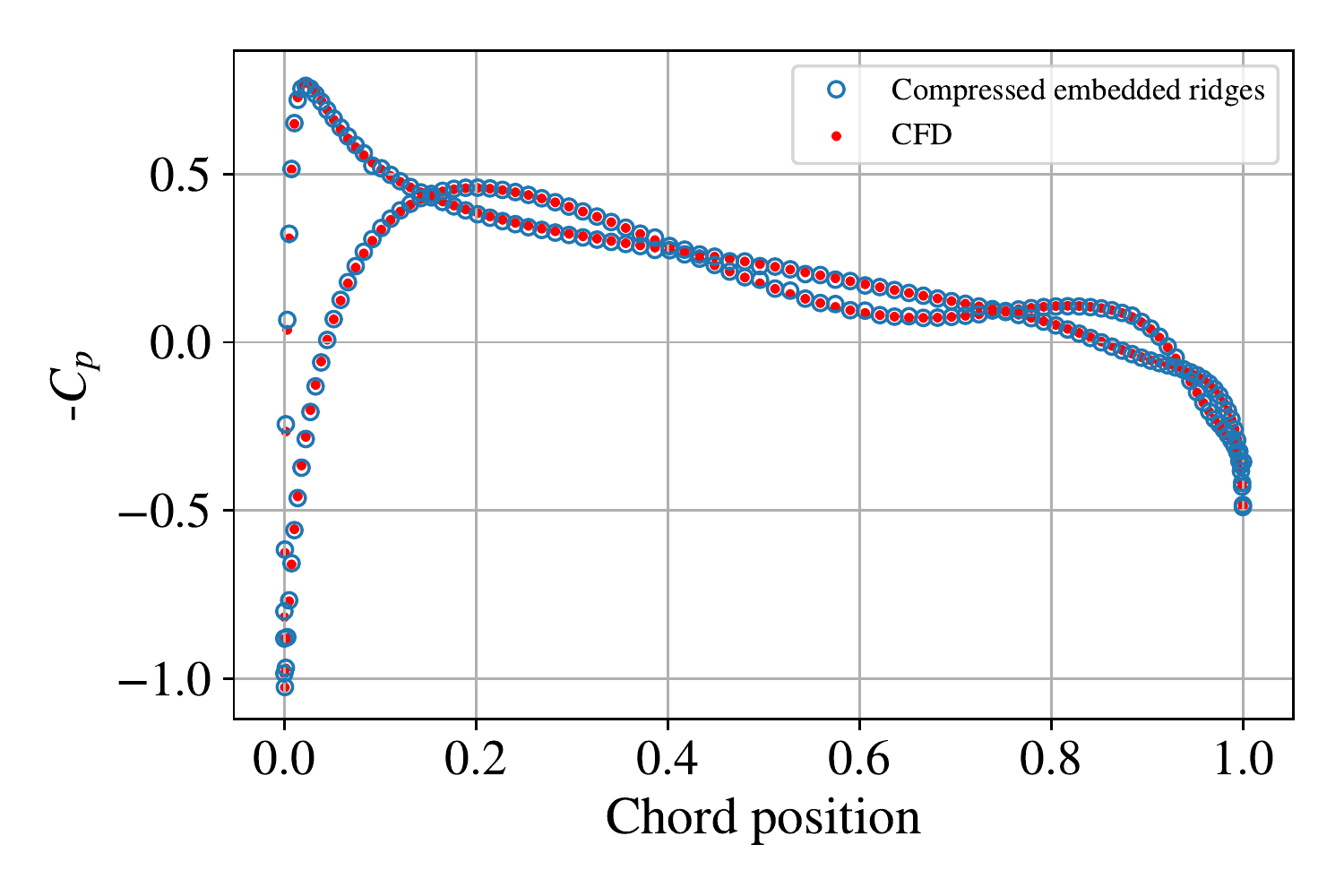}
\caption{Comparing the $C_p$ profile on the surface of the airfoil before and after removing 3000 nodes with ridge compression using an embedded ridge approximation formed from 400 observations.}
\label{fig:Cp_profiles}
\end{figure}

\begin{figure}
\centering
\includegraphics[width=.9\linewidth]{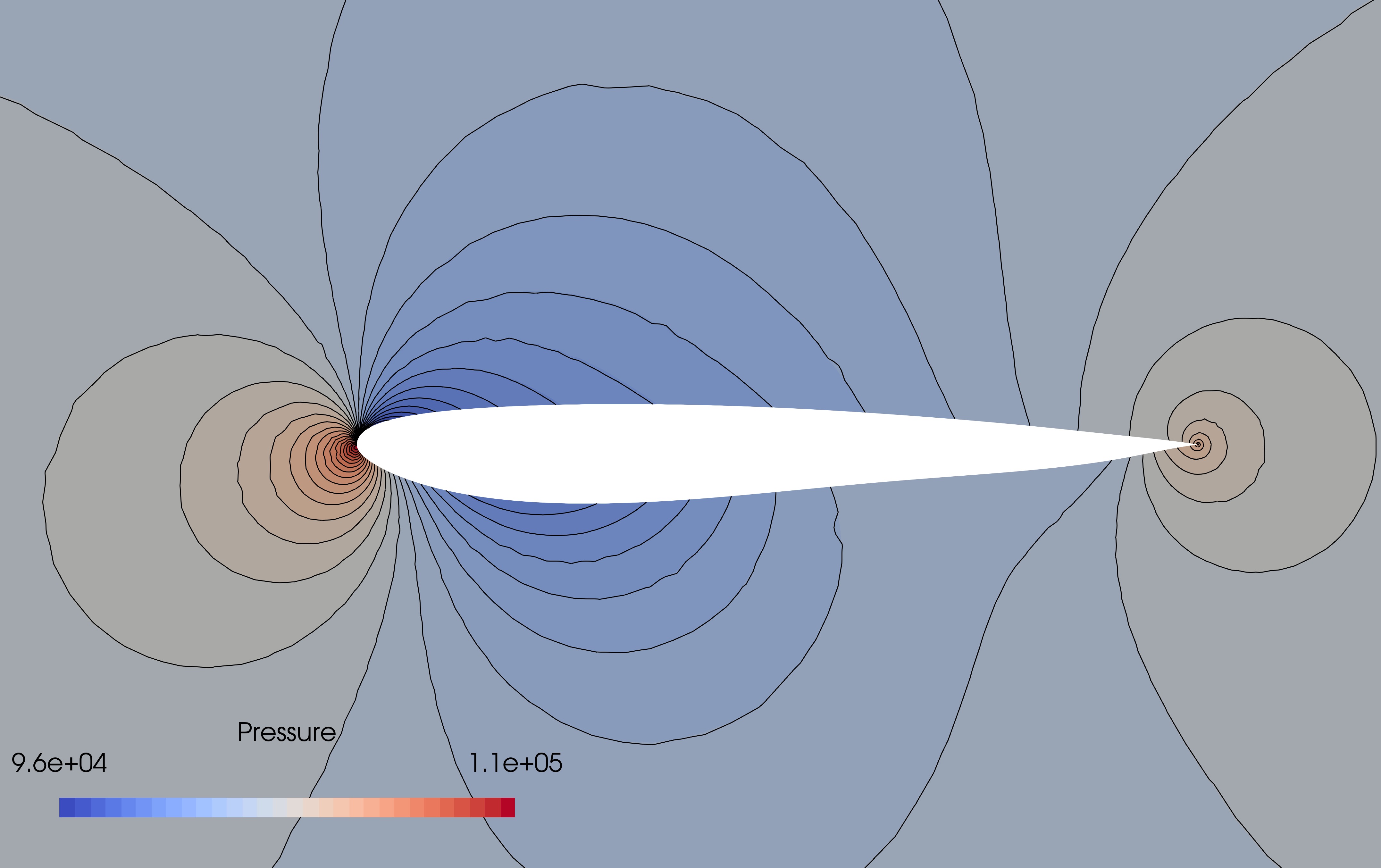}
\caption{Comparison of pressure contours for the flow around a deformed airfoil. Black isolines indicate the estimated flow field using embedded ridge approximation with 400 observations after removing 3000 nodes with ridge compression; colour contours indicate the CFD result.}
\label{fig:contour}
\end{figure}

\begin{figure}
\begin{center}
\begin{minipage}[b]{0.475\linewidth}
\includegraphics[width=1\linewidth]{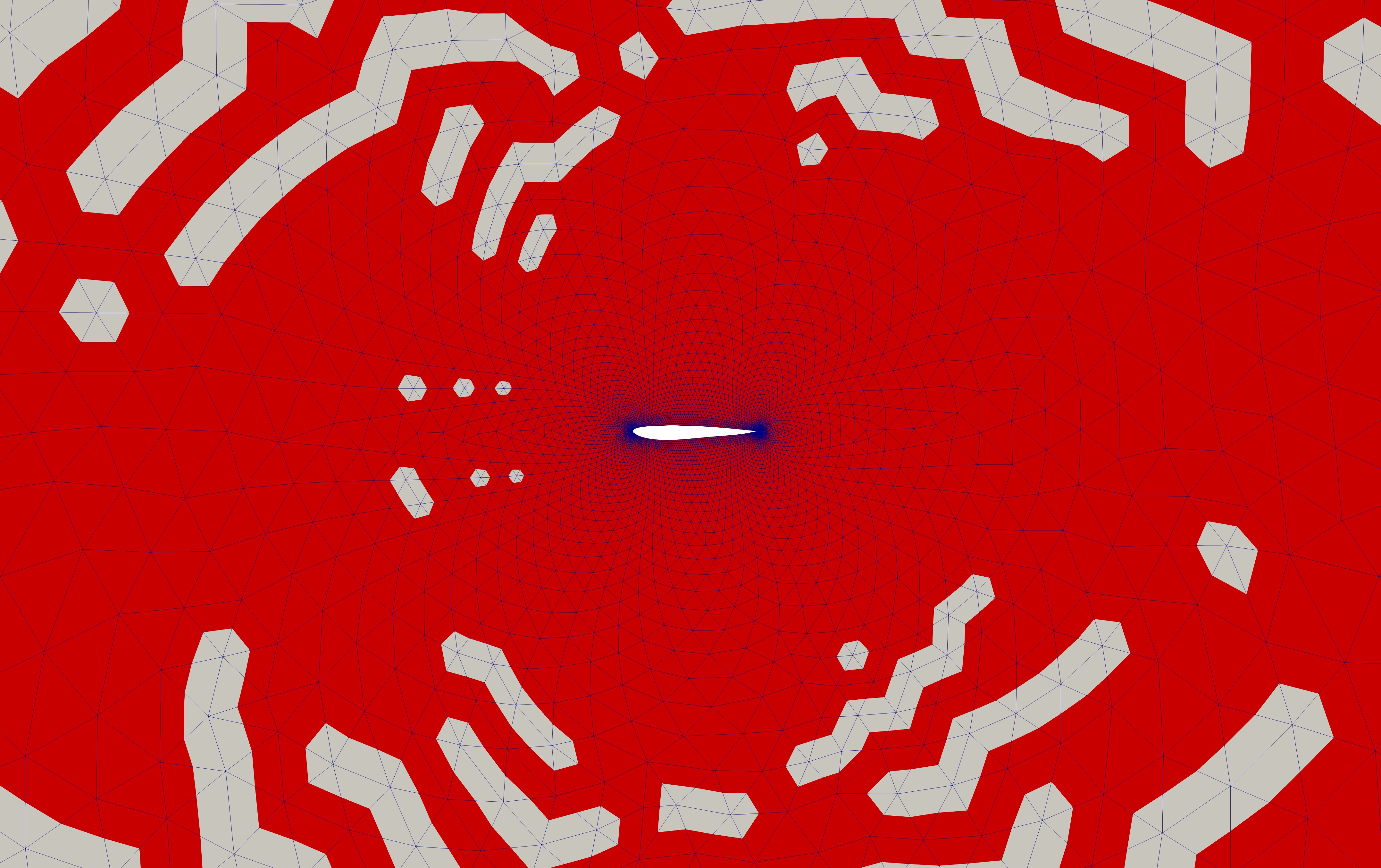}
\end{minipage}
\begin{minipage}[b]{0.475\linewidth}
\includegraphics[width=1\linewidth]{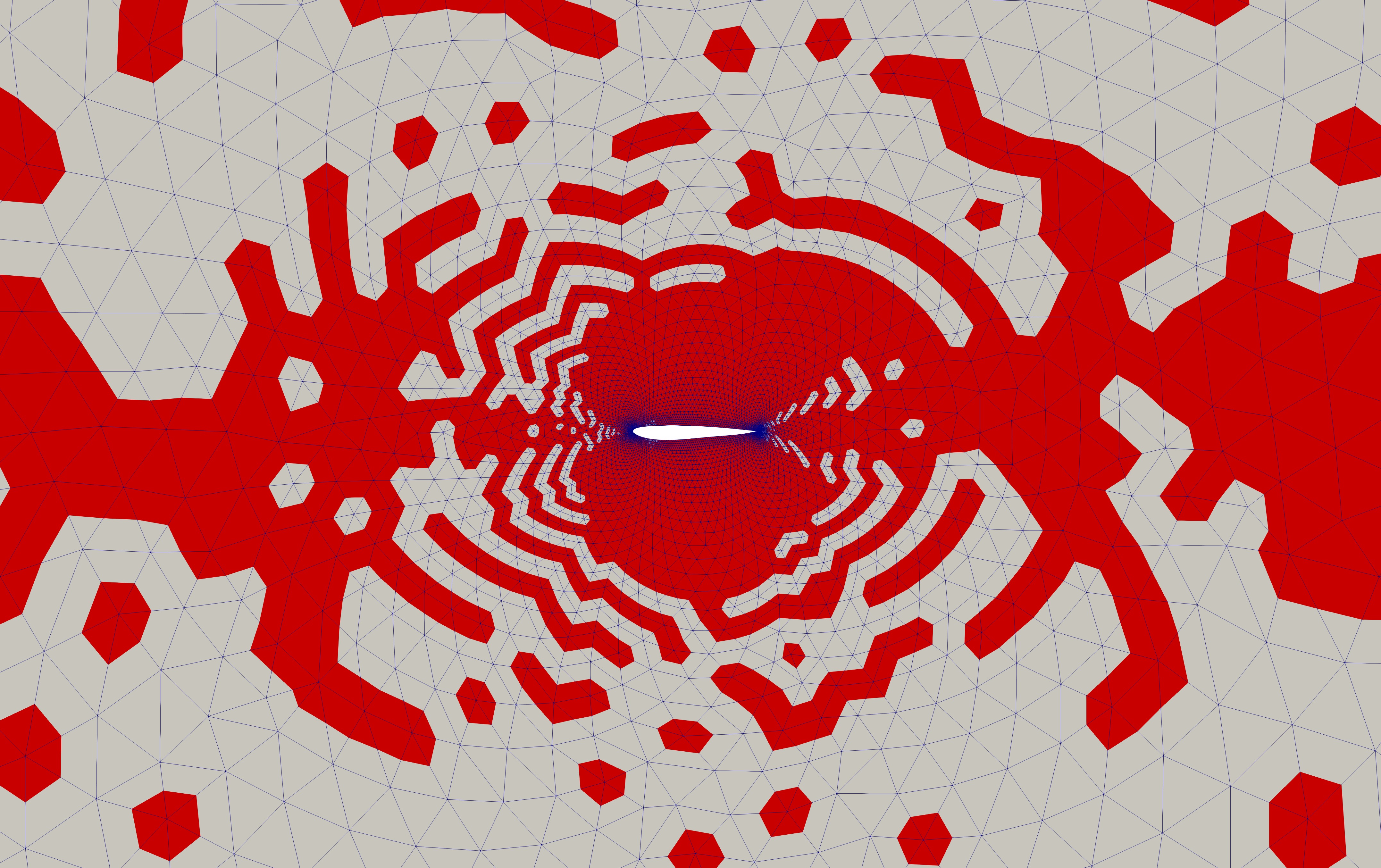}
\end{minipage}
\begin{minipage}[b]{0.475\linewidth}
\includegraphics[width=1\linewidth]{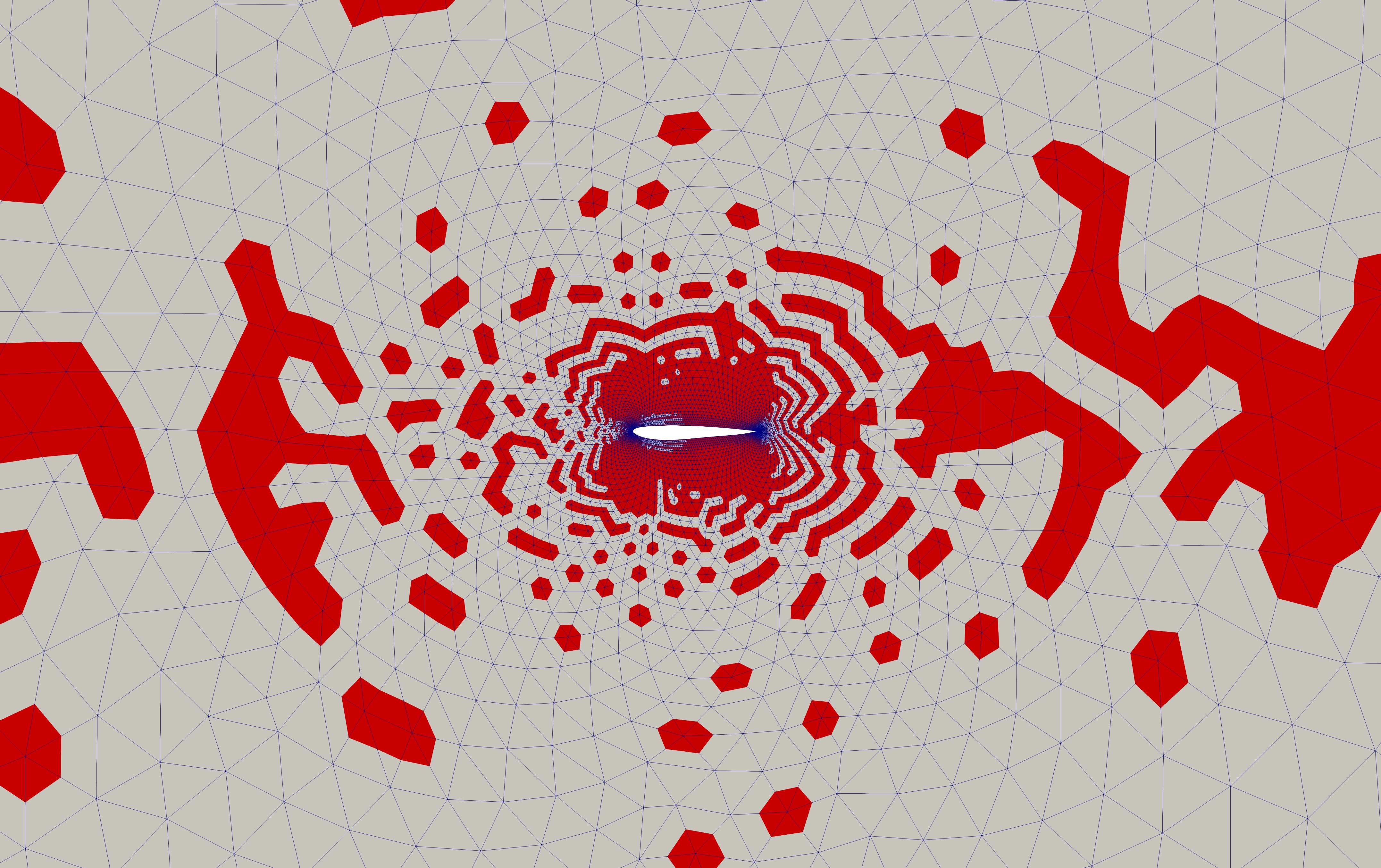}
\end{minipage}
\begin{minipage}[b]{0.475\linewidth}
\includegraphics[width=1\linewidth]{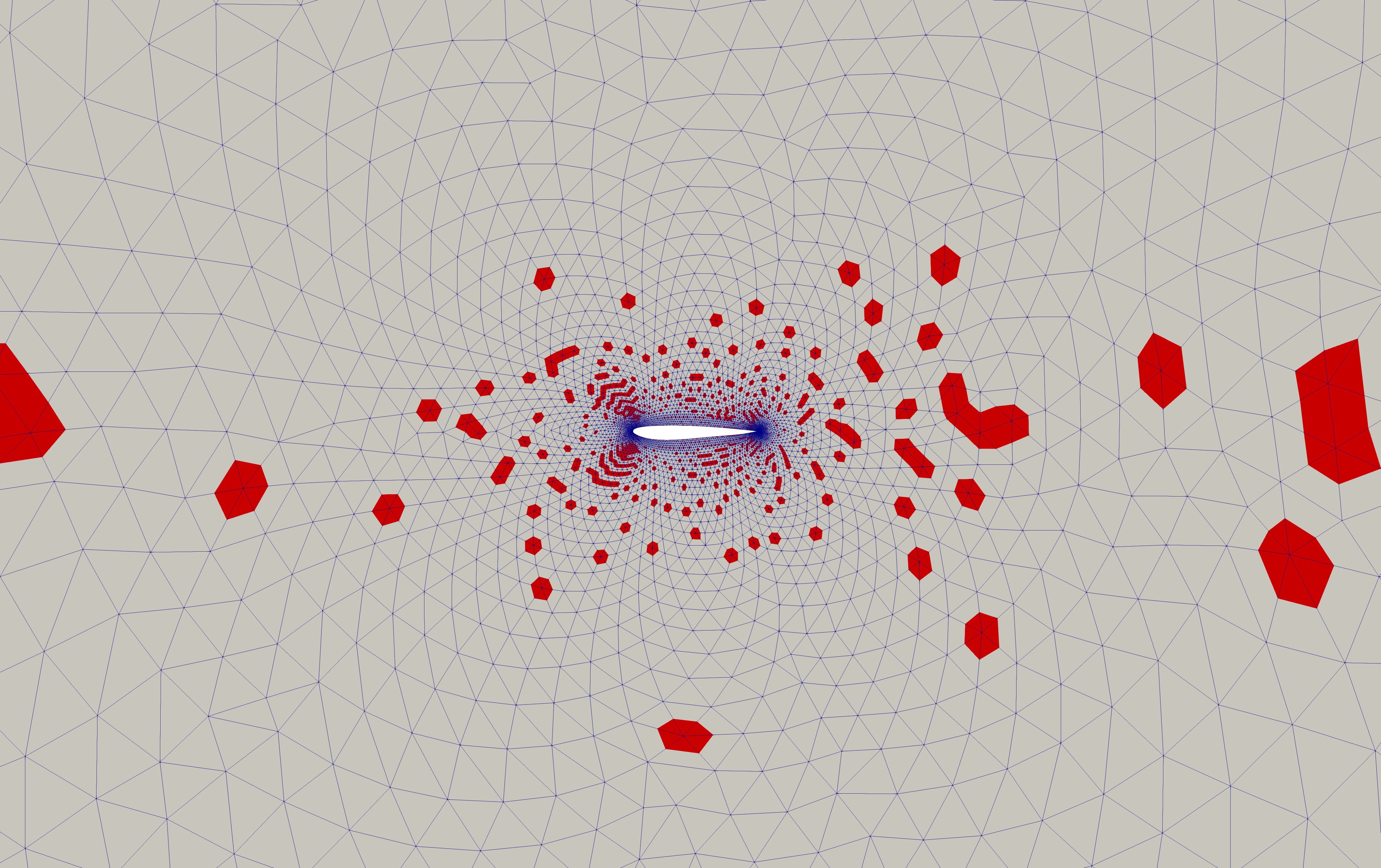}
\end{minipage}
\caption{Ridge compression scheme when 500 (top left), 1500 (top right), 2500 (bottom left) and 4300 (bottom right) nodes are removed. Red nodes are retained nodes, while grey ones are removed.}\label{fig:mesh}
\end{center}
\end{figure}

\section{Conclusions} 
In this paper we introduce the notion of constructing ridge approximations over localised scalar fields, aimed at improved simulation-centric dimension reduction. Our ideas are born from a simple observation: in many PDE-based models, the scalar field at a certain node is only weakly influenced by far-field perturbations. It is more likely to be governed by locally induced perturbations---caused by variations in local boundary conditions or geometry. By interpreting global scalar qois as integrals of these scalar-field quantities, we hypothesise and demonstrate that constructing ridge approximations over individual scalar field nodes instead and then integrating them can reduce the number of computational evaluations required. 